\documentclass[a4paper,11pt]{article}
 \pdfoutput=1
 \usepackage[english]{babel}
 \usepackage{graphicx}
 \usepackage{color}
 \usepackage{amssymb}
 \usepackage{amsmath}
 \usepackage{amsthm}
 \usepackage[round]{natbib}
 \usepackage{subfigure}
 \usepackage[math]{easyeqn}
 \usepackage[mathscr]{eucal}
 \usepackage{comment}
\renewcommand{\(}{$\,}
\renewcommand{\)}{\,$}

\def\nquad{\hspace{-1cm}}
\def\eqdef{\stackrel{\operatorname{def}}{=}}

\newcommand{\cc}[1]{\mathscr{#1}}
\newcommand{\bb}[1]{\boldsymbol{#1}}

\renewcommand{\bar}[1]{\overline{#1}}
\renewcommand{\hat}[1]{\widehat{#1}}
\renewcommand{\tilde}[1]{\widetilde{#1}}

\renewcommand{\Gamma}{\varGamma}
\renewcommand{\Pi}{\varPi}
\renewcommand{\Sigma}{\varSigma}
\renewcommand{\Delta}{\varDelta}
\renewcommand{\Lambda}{\varLambda}
\renewcommand{\Psi}{\varPsi}
\renewcommand{\Phi}{\varPhi}
\renewcommand{\Theta}{\varTheta}
\renewcommand{\Omega}{\varOmega}
\renewcommand{\Xi}{\varXi}
\renewcommand{\Upsilon}{\varUpsilon}

\def\Var{\operatorname{Var}}

\def\argmax{\operatornamewithlimits{argmax}}
\def\argmin{\operatornamewithlimits{argmin}}

\def\R{I\!\!R}
\def\E{I\!\!E}
\def\P{I\!\!P}

\def\kappa{\varkappa}

\def\T{\top}

\def\loc{\operatorname{loc}}

\def\uv{\bb{u}}

\def\Yv{\bb{Y}}

\def\gammav{\bb{\gamma}}

\def\xiv{\bb{\xi}}

\newcommand{\mygraphics}[3]{\begin{center}
    \resizebox{#1\textwidth}{#2\textheight}{\includegraphics{#3}}
    \end{center}
}

\def\ND{\cc{N}}

\def\cond{\, | \,}

\def\rdl{\epsilon}
\def\rd{\bb{\rdl}}
\def\rddelta{\delta}
\def\rdomega{\varrho}

\def\rdb{\rd}
\def\rdm{\underline{\rdb}}

\def\Id{I\!\!\!I}
\def\Ind{\operatorname{1}\hspace{-4.3pt}\operatorname{I}}

\def\La{\mathbb{L}}
\def\Lab{\La_{\rdb}}
\def\Lam{\La_{\rdm}}

\def\DP{D}

\def\gmi{\mathtt{b}}
\def\gmiid{\mathtt{g}_{1}}

\def\ex{\mathrm{e}}

\def\kullb{\cc{K}} 

\def\gm{\mathtt{g}}

\def\yy{\mathtt{y}}
\def\yyc{\yy_{c}}
\def\xx{\mathtt{x}}

\def\alp{\alpha}

\def\riskt{\cc{R}}

\def\bern{q}

\def\rups{\rr_{0}}

\def\K{K}

\def\CS{\cc{E}}

\def\nunu{\nu_{0}}

\def\rdomega{\varrho}

\def\err{\diamondsuit}

\def\LL{\cc{L}}

\def\dimp{p}

\def\thetav{\bb{\theta}}
\def\thetavs{\thetav^{*}}

\def\thetavd{\thetav^{\circ}}

\def\thetas{\theta^{*}}

\def\fs{f}

\def\VP{V}

\def\fis{\mathfrak{a}}

\def\rr{\mathtt{r}}

\def\zz{\mathfrak{z}}

\def\nuu{\mathfrak{u}}
\def\nud{\mathfrak{u}_{0}}

\def\E{\bb{E}}
\def\P{\bb{P}}
\def\ks{k^{*}}
\def\logdens{\ell}
\def\RS{\cc{A}}
\def\Ass{\text{\textsc{A}}}
\def\ER{\E'}
\def\PR{\P'}
\def\VPb{\bar{\VP}}

\newcommand{\citeasnoun}[1]{\cite{#1}}

\newtheorem{theorem}{Theorem}[section]
\newtheorem{lemma}{Lemma}

\newtheorem{corollary}{Corollary}

\renewcommand{\mygraphics}[1]{\includegraphics[width =#1 \textwidth]}

\begin{document}

\title{Local Quantile Regression
\thanks{
    The financial support from the Deutsche
    Forschungsgemeinschaft via SFB 649 ``Economic Risk'',
    Humboldt-Universit\"at zu Berlin is gratefully acknowledged.
    The first author is partially supported by
Laboratory for Structural Methods of Data Analysis in Predictive Modeling, MIPT,
RF government grant, ag. 11.G34.31.0073.}
}
\author{
    Vladimir Spokoiny
    \thanks{Professor, Weierstrass-Institute, Humboldt University Berlin, Moscow Institute of
    Physics and Technology, Mohrenstr. 39, 10117 Berlin, Germany.
    Email:spokoiny@wias-berlin.de},
    Weining Wang
    \thanks{ Research associate at Ladislaus von Bortkiewicz Chair,
    the Institute for Statistics and Econometrics of Humboldt-Universit\"at zu Berlin,
    Spandauer Stra{\ss}e 1,  10178  Berlin, Germany.
    Email:wangwein@cms.hu-berlin.de},
Wolfgang Karl H\"{a}rdle
\thanks{Professor at Humboldt-Universit\"at zu Berlin and Director of
C.A.S.E. - Center for Applied Statistics and Economics, Humboldt-Universit\"at zu
Berlin, Spandauer Stra{\ss}e 1, 10178 Berlin, Germany.
Email:haerdle@wiwi.hu-berlin.de}.
}

\maketitle

\begin{abstract}
  Quantile regression is a technique to estimate conditional quantile curves.
  It provides a comprehensive picture of a response contingent on explanatory variables.
  In a flexible modeling framework, a specific form of the conditional quantile curve is
  not a priori fixed.
  This motivates a local parametric rather than a global
  fixed model fitting approach.
  A nonparametric smoothing estimator of the conditional
  quantile curve requires to balance between local curvature and stochastic
  variability. In this paper, we suggest a local model selection
  technique that provides an adaptive estimator of the conditional
  quantile regression curve at each design point.
  Theoretical results claim that the proposed adaptive procedure performs as good as
  an oracle which would minimize the local estimation risk for the problem at hand.
  We illustrate the performance of the procedure by an extensive simulation study
  and consider a couple of applications:
  to tail dependence analysis for the Hong Kong stock market
  and to analysis of the distributions of the risk factors of temperature dynamics.
\end{abstract}



\section{Introduction}

 Quantile regression is gradually developing into a comprehensive approach for the
 statistical analysis of linear and nonlinear response models. Since the rigorous
 treatment of linear quantile regression by \citeasnoun{Keo:bas:1987}, richer models
 have been introduced into the literature, among them are nonparametric, semiparametric
 and additive approaches. Quantile regression or conditional quantile estimation is a
 crucial element of analysis in many quantitative problems. In financial risk
 management, the proper definition of quantile based Value at Risk impacts asset
 pricing, portfolio hedging and investment evaluation, \citeasnoun{eng:man:2004},
 \citeasnoun{cai:wan:08} and \citeasnoun{fit:wil:08}. In labor market analysis of wage
 distributions, education effects and earning inequalities are analyzed via quantile
 regression. Other applications of conditional quantile studies include, for example,
 conditional data analysis of children growth and ecology, where it accounts for the
 unequal variations of response variables, see  \citeasnoun{ja:ha:10}.

   \begin{figure}
    \begin{center}
      \mygraphics{0.7}{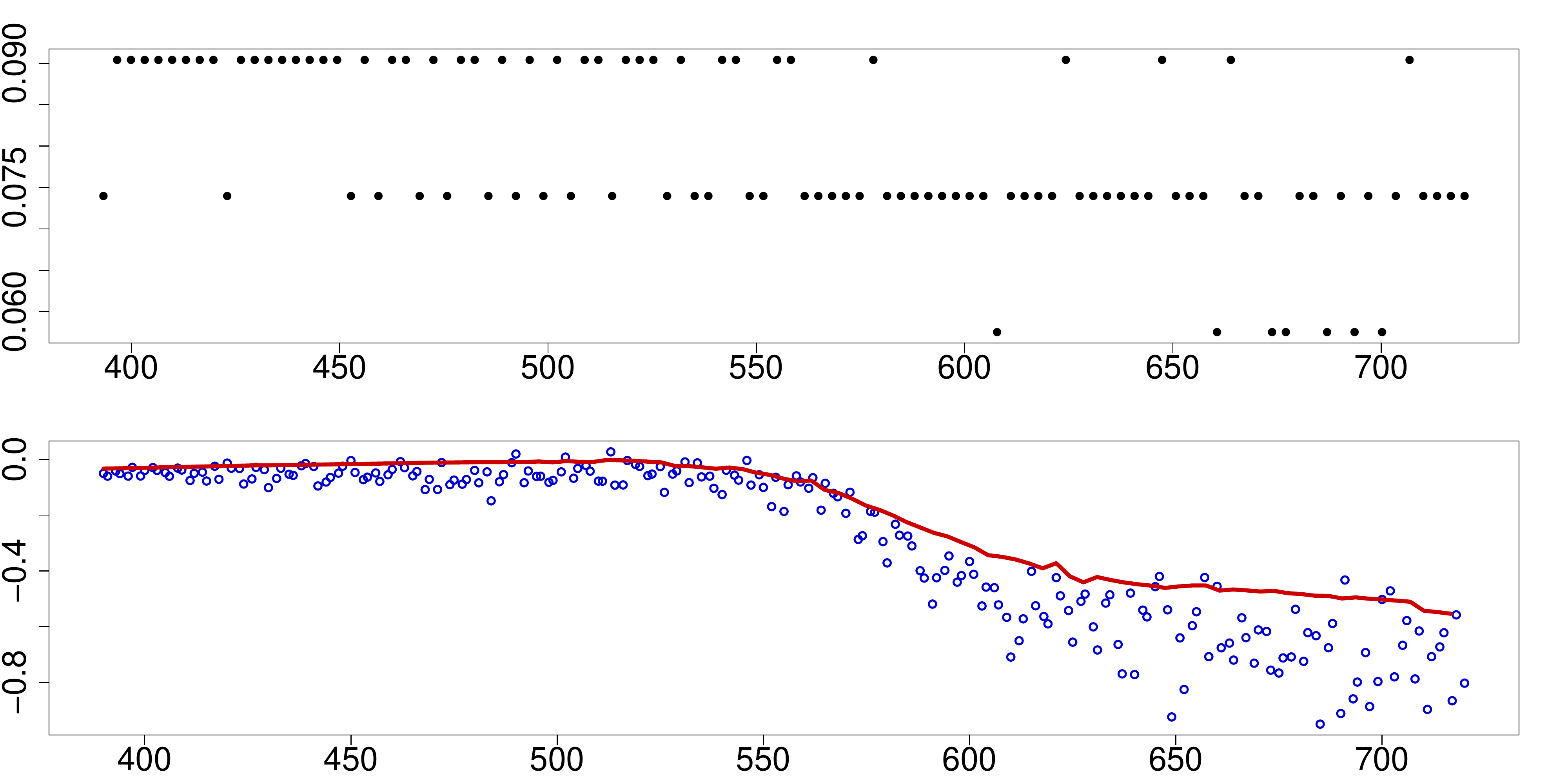}\\
         \caption{The bandwidth sequence (upper panel), plot of data  and the estimated
         \( 90\% \) quantile curve (lower panel)}
      \label{fig:lidar}
    \end{center}
   \end{figure}

  In applications, the predominantly used linear form of the calibrated models is mainly
  determined by practical and numerical reasonings. There are many efficient
  algorithms (like sparse linear algebra and interior point methods) available,
  \citeasnoun{St:Ro:89}, \citeasnoun{St:Ro:97}, \citeasnoun{Jo:Ro:99}, and
  \citeasnoun{Ro:2005}, etc.
  However, the assumption of a linear parametric structure can be too restrictive in
  many applications.
  This observation spawned a stream of literature on nonparametric modeling of
  quantile regression,  \citeasnoun{yu:jo:98}, \citeasnoun{Fan:Hu:94}, etc. One line of
  thought concentrated on different smoothing techniques, e.g. splines, kernel
  smoothing, etc.; see \citeasnoun{fan:1996}. Another line of literature considers
  structural semiparametric models to cope with the curse of dimensionality, like,
  partial linear models, \citeasnoun{ha:so:ya:10}, etc., additive models,
  \citeasnoun{KLX08}, \citeasnoun{jo:lee:05}, etc; single index models,
  \citeasnoun{wu:yu:10}, \citeasnoun{Ro:10}, etc. Yet another strand of literature has
  been involved in ultra-high dimensional situations where a careful variable selection
  technique needs to be implemented, \citeasnoun{al:vi:10} and \citeasnoun{Ro:10}. In
  most of the aforementioned papers on non and semiparametric quantile regression, a
  smoothing parameter selection is implicit, and it is mostly a consequence of
  theoretical assumptions like e.g. degree of smoothness, but falls short in practical
  hints for real data applications. An important exception is the method for local
  nonparametric kernel smoothing by \citeasnoun{yu:jo:98} and \citeasnoun{cai:xu:08}.
  They both propose a data driven bandwidth choice.

  This paper offers a novel data-driven quantile regression procedure. Its numerical
  performance is illustrated by competitive simulation examples and applications to
  real data. The proposed adaptive local quantile regression algorithm is easy to
  implement and works for a wide class of applications. The idea of this algorithm
  is to select the bandwidth locally by a sequence of likelihood ratio tests. We
  also provide a rigorous theoretical study for the proposed method. The optimality
  results are stated as exact and sharp oracle risk bounds. In particular, we show that the
  performance of the adaptive procedure is essentially the same as the best possible
  one. The results apply for finite sample and under mild regularity conditions.

  The main message is that the proposed algorithm is spatially adaptive, stable in
  homogeneous situation and sensitive to structural changes of the quantile curve.
  This conclusion is justified by theoretical results and confirmed by the numerical
  study.
  As an example, consider Figure \ref{fig:lidar} which presents our results for
  analyzing the Lidar data set, \citeasnoun{Ru:2003}. The presented quantile curve
  switches smoothness in the middle, and it is naturally reflected by the bandwidth
  sequence (upper panel) selected. In the presence of changing to sharper slope of
  the curve, the bandwidths get smaller to attain better approximations. This
  example shows that the algorithm proposed in this paper can adaptively choose the
  bandwidth at each design point.

   This article is organized as follows.
   Section~\ref{Sadaptproc} introduces the local model selection (LMS) procedure
   and explains how to important tuning parameters (critical values) can be computed.
   Section~\ref{SsimulLQR} presents a number of Monte Carlo simulations to
   illustrate the proposed methodology.
   In Section~\ref{SapplLQR} the method is applied to check the tail dependency
   among portfolio stocks, and  estimate quantile curves for temperature risk
   factors.
   Section~\ref{StheoryLQR} presents our main theoretical result which states a kind
   of oracle risk bound for the proposed procedure: it performs nearly as good as
   the best one among the considered family of local quantile estimators.
   The necessary conditions and main steps of the proof like ``propagation'',
   ``stability'' and ``oracle'' property are delegated to the Appendix.
   There we also collect some of general results like majorization bounds and
   non-asymptotic Wilks Theorem for the likelihood ratio test statistics.

\section{Adaptive estimation procedure}
\label{Sadaptproc}
This section introduces the considered problem and offers an adaptive estimation
procedure.

\subsection{Quantile regression model}
Given the quantile level \( \tau \in (0,1) \),
the quantile regression model describes the following
relation between the response \( Y \) and the regressor \( X \):
\begin{EQA}[c]
    \P(Y > \fs(x) \cond X=x)
    =
    \tau ,
\label{FYcondx}
\end{EQA}
where \( \fs(x) \) is the unknown \emph{quantile regression function}.
This function is the target of the analysis and it has to be estimated from
independent observations \( \{X_{i}, Y_{i}\}_{i=1}^{n} \).
For the case of a deterministic design,
this quantile relation can be represented as
\begin{EQA}[c]
    Y_{i}
    =
    \fs(X_{i}) + \varepsilon_{i} \, ,
\label{eq:model}
\end{EQA}
where the errors \( \varepsilon_{i} \) follow
\( \P(\varepsilon_{i} > 0) = \tau \).

For simplicity of presentation, we consider a univariate regressor
\( X \in \R^1 \) and a deterministic design in this paper,
an extension to the \( d \)-dimensional case
\( X \in \R^d \) with  \( d>1 \) is straightforward.

\subsection{ A qMLE View on Quantile Estimation}
The quantile function \( \fs(\cdot) \) in (\refeq{eq:model}) is usually recovered by
minimizing the sum
\begin{EQA}[c]
    \sum_{i=1}^{n} \rho_{\tau}\{Y_{i} - \fs(X_{i}) \bigr\} ,
\label{fsrhotau}
\end{EQA}
over the class of all considered quantile functions \( \fs(\cdot) \),
where
\begin{EQA}[c]
    \rho_{\tau}(u)
    \eqdef
    u \{\tau \Ind(u\ge 0) - (1-\tau)\Ind(u<0)\}
    =
    u \bigl\{ \tau - \Ind(u<0) \bigr\}.
\label{rhotauqr}
\end{EQA}
Such an approach is reasonable because the true quantile function
\( \fs(x) \) minimizes the expected value
of the sum in \eqref{fsrhotau}.
An important special case is given by \( \tau = 1/2 \).
Then an estimator of \( \fs(\cdot) \) is built as minimizer of the least absolute
deviations (LAD) contrast \( \sum |Y_{i} - f(X_{i})| \).

The minimum contrast approach based on minimization of \eqref{fsrhotau} can also be put
in a quasi maximum likelihood framework. Assume that the residuals \( \varepsilon_{i} \)
from \eqref{eq:model} are i.i.d. and \( \logdens(x) \) is their negative log-density on \(
\R^{1} \). Then the joint log-density is given by the sum
\begin{EQA}[c]
    - \sum \logdens\{ Y_{i} - \fs(X_{i}) \}
\label{rhodemo}
\end{EQA}
and its maximization is equivalent to minimization of the contrast \eqref{fsrhotau}
with a pdf from
the \emph{asymmetric Laplace distribution} \(ALD_\tau\):
\begin{EQA}[c]
    \logdens(u)
    =
    \logdens_{\tau}(u)
    =
    \log \bigl\{ \tau(1-\tau) \bigr\} - \rho_{\tau}(u),
    \quad
    -\infty < u < \infty .
\label{Laplacerhode}
\end{EQA}
The \emph{parametric approach} (PA)  assumes that the quantile
regression function \( \fs(\cdot) \) belongs to a family of functions
\( \bigl\{ \fs_{\thetav}(x), \, \thetav \in \Theta \bigr\} \),
where \( \Theta \) is a subset of the \( (\dimp+1) \)-dimensional Euclidean space.
Equivalently,
\begin{EQA}[c]
     \fs(x) = \fs_{\thetavs}(x),
     \label{par:all}
\end{EQA}
where \( \thetavs \) is the true parameter which is usually the target of
estimation.

Examples are a constant model:
\begin{EQA}[c]
    \fs_{\thetas}(x) \equiv \theta_{0},
    \label{par:1}
\end{EQA}
with \( \thetas = \theta_{0} \) or a linear model:
\begin{EQA}[c]
    \fs_{\thetavs}(x) = \theta_{0} +  \theta_{1}x,
    \label{par:2}
\end{EQA}
with \( \thetavs = (\theta_{0}, \theta_{1})^{\T} \).

Let \( \P_{\thetav} \) be the parametric measure on the
observation space which corresponds to the regression model
\eqref{eq:model} with \( \fs(\cdot) \equiv \fs_{\thetav}(\cdot) \) and with the
i.i.d. errors \( \varepsilon_{i} \) following the asymmetric Laplace distribution
\eqref{Laplacerhode}.
Then the log-likelihood \( L(\thetav) = L(\Yv,\thetav) \) for \( \P_{\thetav} \)
can be written as
\begin{EQA}
     L(\thetav)
     & \eqdef &
     \log\bigl\{ \tau(1-\tau) \bigr\} \sum_{i=1}^{n} 1
     - \sum_{i=1}^{n} \rho_{\tau} \{Y_{i} -  \fs_{\thetav}(X_{i})\}
     \label{eq:likeliaa}
\end{EQA}
and the qMLE \( \tilde{\thetav} \) maximizes \( L(\thetav) \), or,
equivalently minimizes the contrast
\( \sum_{i=1}^{n} \rho_{\tau} \{Y_{i} -  \fs_{\thetav}(X_{i})\} \) over all
\( \thetav \in \Theta \).

The described parametric construction is based on two assumptions: one is about the
error distribution \eqref{Laplacerhode} and the other one is about the shape of the
regression function \( \fs \).
 However, it is only used for motivating our approach.
%
Our theoretical study will be done under the true data distribution which follows
\eqref{eq:model} under mild regularity conditions.
The next section explains how a smooth regression
function \( \fs \) can be modeled by a flexible local parametric assumption.

\subsection{Local polynomial qMLE}
This section explains how the restrictive \emph{global PA}
\( \fs(\cdot) \equiv \fs_{\thetavs}(\cdot) \) can be relaxed by using a local
parametric approach.
Let a point \( x \) be fixed.
The \emph{local PA} at a point \( x \in \R \) only requires that the quantile
regression function \( \fs(\cdot) \) can be approximated by a parametric function
\( \fs_{\thetav}(\cdot) \) from the given family in a vicinity of \( x \).
Below we fix a family of polynomial functions of degree \( \dimp \) motivated by Taylor approximation:
\begin{EQA}[c]
    \fs(u)
    \approx
    \fs_{\thetav}
    \eqdef
    \theta_{0} +  \theta_{1}(u-x) + \ldots + \theta_{p}(u-x)^{p}/p!
    \label{eq:locpoly}
\end{EQA}
for \( \thetav = (\theta_{0},\ldots,\theta_{p})^{\T} \).
The corresponding parametric model can be written as
\begin{EQA}[c]
    Y_{i} = \Psi_{i}^{\T} \thetav + \varepsilon_{i} \, ,
\label{modparlin}
\end{EQA}
where \( \Psi_{i} = \{ 1, (X_{i} - x), (X_{i} - x)^{2}/2!,\ldots,(X_{i} - x)^{\dimp}/\dimp!\}^{\T}
\in \R^{\dimp+1} \).

A \emph{local likelihood approach} at \( x \) is specified by a
\emph{localizing scheme} \( W \) given by a collection of weights \( w_{i} \)
for \( i=1,\ldots,n \).
The weights \( w_{i} \) vanish for points \( X_{i} \) lying outside a vicinity of
the point \( x \).
A standard proposal for choosing the weights \( W \) is
\( w_{i} = K_{\loc}\{(X_{i}-x)/h\} \), where \( K_{\loc}(\cdot) \) is a \emph{kernel function}
with a compact support,
while \( h \) is a \emph{bandwidth} controlling the degree of localization.

Define now the local log-likelihood at \( x \) by
\begin{EQA}[c]
     L(W,\thetav)
     \eqdef
     \log\tau(1-\tau) \sum_{i=1}^{n} w_{i}
     - \sum_{i=1}^{n} \rho_{\tau}(Y_{i} - \Psi_{i}^{\T} \thetav) w_{i} \, .
     \label{likelihood}
\end{EQA}
This expression is similar to
the global log-likelihood in \eqref{eq:likeliaa}, but each summand in \( L(W,\thetav) \)
is multiplied with the weight \( w_{i} \), so only the points from the local vicinity
of \( x \) contribute to \( L(W,\thetav) \).
Note that this local log-likelihood depends on the central point \( x \) via the
structure of the basis vectors \( \Psi_{i} \) and via the weights \( w_{i} \).
The corresponding local qMLE at \( x \) is defined via maximization of
\( L(W,\thetav) \):
\begin{EQA}
\label{eq:non}
    \tilde{\thetav}(x)
    &=&
    \{\tilde{\theta}_{0}(x),\tilde{\theta}_{1}(x),\ldots,\tilde{\theta}_{p}(x)\}^\T
    \\
    & \eqdef &
    \argmax_{\thetav \in \Theta} L(W,\thetav)
    \\
    &=&
    \argmin_{\thetav \in \Theta}
    \sum_{i=1} \rho_{\tau}(Y_{i} - \Psi_{i}^{\T} \thetav) w_{i} \, .
\end{EQA}
The first component \( \tilde{\theta}_{0}(x) \) provides an
estimator of \( \fs(x) \), while \( \tilde{\theta}_{m}(x) \) is an estimator of the
derivative \( \fs^{(m)}(x) \), \( m=1,\ldots,\dimp \).

\subsection{Selection of a Pointwise Bandwidth}
\label{Spointband}

The choice of bandwidth \( h \) is an important issue in implementing
(\refeq{eq:non}). One can reduce the variance of the estimation by increasing the
bandwidth, but at a price of possibly inducing more modeling bias
measured by the accuracy of approximation in \eqref{eq:locpoly}; see
Figure~\ref{fig:lmspic}.

A desirable choice of a bandwidth at a fixed point would strike a balance between
the variance and the bias depending on the local shape of \( \fs(\cdot) \) in the
vicinity of \( x \).
Many approaches have been proposed along this line; {see e.g. \cite{yu:jo:98}} and
references therein.
However, their justification and implementation is based on
asymptotic arguments and require large samples.
Here we propose a pointwise bandwidth selection technique based on a finite sample theory.

Our basic setup of the algorithm is described as follows.
First one fixes a finite ordered set of possible bandwidths
\( h_{1} < h_{2} < \ldots < h_{\K} \), where \( h_{1} \) is very small, while
\( h_{\K} \) should be a global bandwidth of the order of the design range.
The bandwidth sequence can be taken geometrically increasing of the form
\( h_{k} = a b^{k} \) with fixed \( a > 0 \), \( b > 1 \), and
\(  n^{-1}<a b^{k}<1 \) for \( k = 1, \ldots, \K \) (\( \Ass. 2.\)).
The total number \( \K \) of the candidate bandwidths
is then at most logarithmic in the sample size \( n \).
%
For each \( k \le \K \), an
ordered weighting schemes
\( W^{(k)} =( w^{(k)}_{1}, w^{(k)}_{2}, \ldots, w^{(k)}_{n})^{\T} \) is defined via
\( w_{i}^{(k)} \eqdef K_{\loc}\{(x-X_{i})/h_{k}\} \) leading
to a local quantile estimator \( \tilde{\thetav}_{k}(x) \)
with
\begin{EQA}[c]
    \tilde{\thetav}_{k}(x)
    =
    \argmax_{\thetav \in \Theta} L(W^{(k)},\thetav)
    =
    \argmin_{\thetav \in \Theta}
    \sum_{i=1} \rho_{\tau}(Y_{i} - \Psi_{i}^{\T} \thetav) w_{i}^{(k)} \, .
\label{eq:nonk}
\end{EQA}
%
%
The proposed selection procedure is similar in spirit to \citeasnoun{le:ma:96}.
If the underlying quantile regression function is smooth, one can expect a good
quality of approximation \eqref{eq:locpoly} for a large bandwidth among
\( \{ h_{k} \}_{k=1}^{\K} \).
Moreover, if the approximation is good for one bandwidth, it will be also suitable
for all smaller bandwidths.
So, if we observe a significant difference between the estimator
\( \tilde{\thetav}_{k}(x) \) corresponding to the bandwidth \( h_{k} \) and an
estimator \( \tilde{\thetav}_{\ell}(x) \) corresponding to a smaller bandwidth \( h_{\ell} \),
this is an indication that the approximation \eqref{eq:locpoly} for the window size
\( h_{k} \) becomes too rough.
This justifies the following algorithm.
Start with the smallest bandwidth \( h_{1} \).
For any \( k > 1 \), compute
the local qMLE \( \tilde{\thetav}_{k}(x) \) and check whether it is consistent with
all the previous estimators \(\tilde{\thetav}_{\ell}(x) \) for \( \ell < k \). If the
consistency check is negative, the procedure terminates and selects the latest
accepted estimator.

The most important ingredient of the method is the consistency check.
The Lepski method suggests to use the difference
\( \tilde{\thetav}_{k}(x) - \tilde{\thetav}_{\ell}(x) \) as a test statistic;
see e.g. \cite{le:ma:96}.
We follow the suggestion from \cite{Polzehl:2006} and apply a localized likelihood
ratio type test.
More precisely, the local MLE \( \tilde{\thetav}_{\ell}(x) \)
maximizes the log-likelihood \( L\bigl( W^{(\ell)},\thetav \bigr) \), and the maximal
value of \eqref{likelihood} given by
\( \sup_{\thetav} L\bigl( W^{(\ell)},\thetav \bigr)
= L\bigl( W^{(\ell)},\tilde{\thetav}_{\ell}(x) \bigr) \)
is compared with the particular log-likelihood value
\( L\bigl( W^{(\ell)},\tilde{\thetav}_{k}(x) \bigr) \), where the estimator
\( \tilde{\thetav}_{k}(x) \)
is obtained by maximizing the other local log-likelihood function
\( L(W^{(k)},\thetav) \).
The difference
\( L\bigl( W^{(\ell)},\tilde{\thetav}_{\ell}(x) \bigr)
- L\bigl( W^{(\ell)},\tilde{\thetav}_{k}(x) \bigl) \) is
always non-negative.
The check rejects \( \tilde{\thetav}_{k}(x) \) if this difference is too large
for some \( \ell < k \).
Equivalently one can say that the test checks whether \( \tilde{\thetav}_{k}(x) \)
belongs to the confidence sets \( \CS_{\ell}(\zz) \) of \( \tilde{\thetav}_{\ell}(x) \):
\begin{EQA}[c]
    \CS_{\ell}(\zz)
    \eqdef
    \bigl\{
        \thetav \in \Theta:
        L\bigl( W^{(\ell)},\tilde{\thetav}_{\ell}(x) \bigr)
        - L\bigl( W^{(\ell)},\thetav \bigr)
        \le
        \zz
    \bigr\} .
\label{CSellzz}
\end{EQA}
A great advantage of the likelihood ratio test is that the critical value
\( \zz \) can be selected universally.
This is justified by the Wilks phenomenon: the likelihood ratio test statistics is
nearly \( \chi^{2} \) and its asymptotic distribution depends only on the dimension
of the parameter space.
Unfortunately, these arguments do not apply to finite samples under possible
model misspecification and we therefore offer an alternative way of fixing the critical values
\( \zz \) which is based on the so called propagation condition.
We also allow that the width of the confidence set \( \CS_{\ell}(\zz) \) depends on
the index \( \ell \), that is, \( \zz = \zz_{\ell} \).
Our adaptation algorithm can be summarized as follows:
at each step \( k \), an estimator \( \hat{\thetav}_{k}(x) \) is
constructed based on the first \( k \) estimators
\( \tilde{\thetav}_{1}(x),\ldots, \tilde{\thetav}_{k}(x) \) by the following rule,
\begin{itemize}
    \item
     Start with \( \hat{\thetav}_{1}(x) = \tilde{\thetav}_{1}(x) \).
    \item
     For \( k \ge 2 \), \( \tilde{\thetav}_{k}(x) \) is accepted and
     \( \hat{\thetav}_{k}(x) \eqdef \tilde{\thetav}_{k}(x) \), if
     \( \tilde{\thetav}_{k-1}(x) \) was accepted and
\begin{EQA}[c]
    L\bigl( W^{(\ell)},\tilde{\thetav}_{\ell}(x) \bigr)
    - L\bigl( W^{(\ell)},\tilde{\thetav}_{k}(x) \bigr)
    \le
    \zz_{\ell},
    \qquad
    \ell = 1,\ldots,k-1 .
\label{eq:confs}
\end{EQA}
     \item
        Otherwise \( \hat{\thetav}_{k}(x) = \hat{\thetav}_{k-1}(x) \).
  \end{itemize}

The adaptive estimator \( \hat{\thetav}(x) \) is the latest accepted estimator
        after all \( \K \) steps:
\begin{EQA}[c]
    \hat{\thetav}(x)
    \eqdef
    \hat{\thetav}_{\K}(x)
\label{hatthetKx}
\end{EQA}
A visualization of the procedure is presented in Figure \ref{fig:lmspic}.
The critical values \( \zz_{\ell} \)'s are selected by an algorithm based on
the propagation condition explained in the next section.


 \begin{figure}
   \begin{center}
     \mygraphics{0.5}{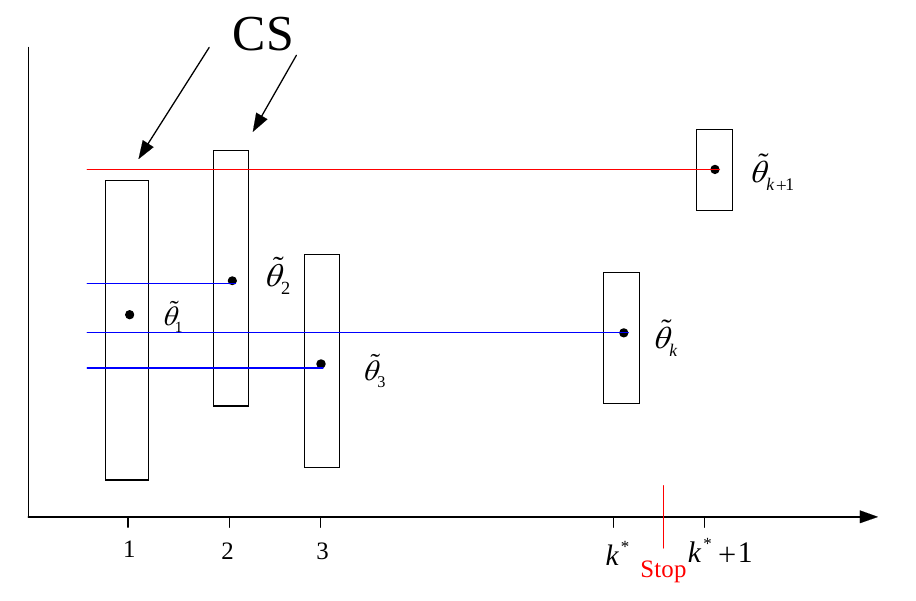}\\
     \caption{Demonstration of the local adaptive algorithm.}\label{fig:lmspic}
  \end{center}
 \end{figure}

\subsection{ Parameter Tuning by Propagation Condition}
\label{SpropCV}
The practical implementation requires to fix the critical values of
\( \zz_{1},\ldots,\zz_{\K-1} \).
We apply the \emph{propagation} approach which is an extension of the proposal from
\cite{Spokoiny:2009,SV2007}.
The idea  is to tune the parameter of the procedure for one artificial
parametric situation.
Later we show that such defined critical values work well in the general setup and
provide a nearly efficient estimation quality.
 The
presented bandwidth selector can be viewed as a multiple testing procedure.
This suggests fixing the critical values as in the general testing theory
by ensuring a prescribed performance under the null hypothesis.
In our case, the null hypothesis corresponds to the pure parametric situation with
\( \fs(\cdot) \equiv \fs_{\thetavs}(\cdot) \) in the equation \eqref{eq:model}.
Moreover, we fix some particular distribution of the errors \( \varepsilon_{i} \),
our specific choice is \(ALD_\tau\) with
parameter \( \tau \).
Below in this section we denote by \( \P_{\thetavs} \) the data distribution
under these assumptions.

For this artificial data generating process,
all the estimators \( \tilde{\thetav}_{k}(x) \) should be consistent to
each other and the procedure should not terminate at any intermediate step
\( k < \K \).
This effect is called as \emph{propagation}: in the parametric situation, the degree
of locality will be successfully increased until it reaches the largest scale.
The critical values are selected to ensure the desired \emph{propagation condition}
which effectively means a ``no false alarm'' property: the selected adaptive
estimator coincides in the most cases with the estimator \( \tilde{\thetav}_{\K}(x) \)
corresponding to the largest bandwidth.
The event \( \bigl\{ \tilde{\thetav}_{k}(x) \neq \hat{\thetav}_{k}(x) \bigr\} \)
for \( k \le \K \)
is associated with a false alarm and the corresponding loss can be measured by the
difference
\begin{EQA}[c]
    L\bigl( W^{(k)},\tilde{\thetav}_{k}(x),\hat{\thetav}_{k}(x) \bigr)
    \eqdef
    L\bigl( W^{(k)},\tilde{\thetav}_{k}(x) \bigr)
    - L\bigl( W^{(k)},\hat{\thetav}_{k}(x) \bigr) .
\label{LWktt}
\end{EQA}
The \emph{propagation condition} postulates that
the risk induced by such false alarms is smaller than the upper bound for the risk
of the estimator \( \tilde{\thetav}_{k}(x) \) in the pure parametric situation:
\begin{EQA}[c]
    \E_{\thetavs} L^{r}\bigl( W^{(k)},\tilde{\thetav}_{k}(x),\hat{\thetav}_{k}(x) \bigr)
    \le
    \alpha \riskt_{r} 
    \qquad
    k=2,\ldots,K,
\label{eq:propa}
\end{EQA}
where the constant \( \riskt_{r} \) is such that for all \( k \le \K \), it holds
\begin{EQA}[c]
        \E_{\thetavs} L^{r}\bigl( W^{(k)},\tilde{\thetav}_{k}(x),\thetavs \bigl)
    \le
    \riskt_{r} \, .
\label{riskkx}
\end{EQA}
The values
\( \alpha \) and \( r \) in \eqref{eq:propa} are two hyper-parameters.
The role of \( \alpha \) is similar to the significance level of a test,
while \( r \) denotes the power of the loss function.
It is worth mentioning that
\begin{EQA}[c]
     \E_{\thetavs} L^{r}\bigl( W^{(k)},\tilde{\thetav}_{k}(x),\hat{\thetav}_{k}(x) \bigl)
     \to
     \P_{\thetavs} \bigl\{ \tilde{\thetav}_{k}(x) \neq \hat{\thetav}_{k}(x) \bigr\},
     \qquad
     r \to 0 .
\end{EQA}
The critical values \( \zz_{1},\ldots,\zz_{\K-1} \) enter implicitly
in the propagation condition: if
the false alarm event \( \{ \tilde{\thetav}_{k}(x) \neq \hat{\thetav}_{k}(x) \} \)
happens too often, it is an indication that some of the critical values
\( \zz_{1},\ldots,\zz_{k-1} \) are too small.
%
Note that (\refeq{eq:propa}) relies on the artificial parametric model
\( \P_{\thetavs} \) instead of the true model \( \P \).
The point \( \thetavs \) here can be selected arbitrarily, e.g. \( \thetavs = 0 \).
This fact relies on the linear parametric structure of the model \eqref{modparlin} and is
justified by the following simple lemma.

\begin{lemma}
\label{Lpivotq}
The distribution of \( L\bigl( W^{(k)},\tilde{\thetav}_{k}(x),\hat{\thetav}_{k}(x) \bigr) \)
and of \( L\bigl( W^{(k)},\tilde{\thetav}_{k}(x),\thetavs \bigr) \) under \( \P_{\thetavs} \)
does not depend on \( \thetavs \).
\end{lemma}

\begin{proof}
Under PA \( \fs(\cdot) \equiv \fs_{\thetavs}(\cdot) \), it holds
\( Y_{i} - f(X_{i}) = Y_{i} - \Psi_{i}^{\T}\thetavs = \varepsilon_{i} \) and hence,
\begin{EQA}[c]
    L(W^{(k)},\thetav)
    =
    \log\{\tau(1-\tau)\}\sum_{i=1}^{n} w_{i}^{(k)}
    + \sum_{i=1}^{n} \rho_{\tau}\{ \varepsilon_{i} - \Psi_{i}^{\T} (\thetav - \thetavs) \}w_{i}^{(k)} .
\label{LWkts}
\end{EQA}
A simple inspection of this formula yields that
the distribution of \( L(W^{(k)},\thetav) \) only depends
on \( \uv = \thetav - \thetavs \).
In other words, we can use the free parameter \( \uv = \thetav - \thetavs \)
whatever \( \thetavs \) is, e.g. \( \thetavs \equiv 0 \).
The same argument applies to the difference
\( L\bigl( W^{(k)},\tilde{\thetav}_{k}(x),\tilde{\thetav}_{\ell}(x) \bigr) \) for
\( \ell < k \).
Moreover, \( L\bigl( W^{(k)},\tilde{\thetav}_{k}(x),\hat{\thetav}_{k}(x) \bigr) \) is a function of
\( \bigl\{ L\bigl( W^{(k)},\tilde{\thetav}_{k}(x),\tilde{\thetav}_{\ell}(x) \bigr)
\bigr\}^{k}_{\ell=1}\),
so the distribution of
\( L\bigl( W^{(k)},\tilde{\thetav}_{k}(x),\hat{\thetav}_{k}(x) \bigr) \)
does not depend on \(\thetavs\).
\end{proof}

A choice of critical values \( \zz_{1} ,\ldots,\zz_{\K-1} \)
can be implemented in the following way:
   \begin{itemize}
        \item Consider first only \( \zz_{1} \) and fix
        \( \zz_{2}=\ldots=\zz_{\K-1}=\infty \), leading to the estimators
        \( \hat{\thetav}_{k}(\zz_{1}, x) \) for \( k = 2,\ldots, \K \).
        The value \( \zz_{1} \) is selected as the minimal one for which
\begin{EQA}[c]
    \frac{1}{\riskt_{r}} \E_{\thetavs}
    L^{r}\bigl( W^{(k)},\tilde{\thetav}_{k}(x),\hat{\thetav}_{k}(\zz_{1},x) \bigr)
    \le
    \frac{\alpha}{\K-1},
    \qquad
    k = 2,\ldots,\K.
\label{zz1prop}
\end{EQA}

        \item With selected \( \zz_{1},\ldots,\zz_{k-1} \), set
        \( \zz_{k+1}=\ldots=\zz_{\K-1}=\infty \).
        Any particular value of \( \zz_{k} \) would lead to the set of parameters
        \( \zz_{1},\ldots,\zz_{k},\infty,\ldots,\infty \) and the family of estimators
        \( \hat{\thetav}_m(\zz_{1},\ldots,\zz_{k},x) \) for \( m=k+1,\ldots,\K \).
        Select the smallest \( \zz_{k} \) ensuring
\begin{EQA}[c]
    \frac{1}{\riskt_{r}} \E_{\thetavs}
    L^{r}\bigl(
        W^{(m)},\tilde{\thetav}_{m}(x),\hat{\thetav}_{m}(\zz_{1},\zz_{2},\ldots,\zz_{k},x)
    \bigr)
    \le \frac{k \alpha}{\K-1}
    \qquad
\label{zzkprop}
\end{EQA}
for all \( m = k+1,\ldots,\K \).
\end{itemize}

Few remarks to the proposed algorithm.
\begin{enumerate}
  \item A value \( \zz_{1} \) ensuring \eqref{zz1prop} always exists because
  the choice \( \zz_{1} = \infty \) yields
  \( \hat{\thetav}_{k}(\zz_{1},x) = \tilde{\thetav}_{k}(x) \) for all \( k \ge 2 \).

  \item The value \( L^{r}\bigl( W^{(m)},\tilde{\thetav}_{m}(x),
  \hat{\thetav}_{m}(\zz_{1},\zz_{2},\ldots,\zz_{k},x) \bigr) \) from \eqref{zzkprop}
  only accumulates the losses associated with the false alarms at the first \( k \) steps
  of the procedure.
  The other checks at further steps are always accepted because the corresponding
  critical values \( \zz_{k+1},\ldots \zz_{\K-1}\) are set to infinity.

  \item The accumulated risk bound \( \frac{k\alpha}{\K-1} \) grows at each step by
  \( \alpha/(\K-1) \).
  This value can be seen as maximal risk associated with the CV's
  \( \zz_{1},\zz_{2},\ldots,\zz_{k} \).

  \item The value \( \zz_{k} \) ensuring \eqref{zzkprop} always exists, because
  the choice \( \zz_{k} = \infty \) yields
\begin{EQA}[c]
    \hat{\thetav}_{m}(\zz_{1},\zz_{2},\ldots,\zz_{k},x)
    =
    \hat{\thetav}_{m}(\zz_{1},\zz_{2},\ldots,\zz_{k-1},x)
\end{EQA}
  for all \( m \ge k \).

  \item
  All the computed values depend on the considered linear parametric model, the
  sequence of bandwidths \( h_{k} \) and the quantile level \( \tau \). They also
  depend on the local point \( x \) via the basis vectors \( \Psi_{i} \). However,
  under common regularity conditions on the design \( X_{1},\ldots,X_{n} \), the
  dependency on \( x \) is rather minor. Therefore, the adaptive estimation procedure can be
  repeated at different points without reiterating the steps  of selecting the
  critical values.

\end{enumerate}

\section{Simulations}
\label{SsimulLQR}
First, we check the critical values at different quantile levels \( \tau = 0.05,
0.5, 0.75, 0.95 \) and for different noise distributions: a) ALD, b) normal
and c) student \( t(3) \).  We also study how  misidentification
of noise distribution affects the critical values.

Second, we compare the performance of our local bandwidth algorithm with two other
bandwidth selection techniques. One proposal is from \citeasnoun{yu:jo:98}, in which
they consider a rule of thumb bandwidth based on the assumption that the quantiles
are parallel, and another comes from \citeasnoun{cai:xu:08}, where an approach based
on a nonparametric version of the Akaike information criterion (AIC) is implemented.

\subsection{Critical Values}
Here we summarize our numerical results on choosing the critical values by the
propagation condition as described in Section~\ref{SpropCV}.
We only consider local constant modeling with \( \dimp = 0 \) and local linear
modeling with \( \dimp=1 \) starting with \( \dimp = 0 \).

Table \ref{tb:tau} shows the critical values with several choices of \( \alpha \)
and \( r \) with \( \tau = 0.75 \), \( m = 10000 \)  Monte Carlo samples, and an
bandwidth sequence \( (8, 14, 19, 25, 30, 36, 41, 52)*0.001 \) scaled to the
interval \( [0,1] \).
Critical values decrease when \( \alpha \) increases, and
increase when \( r \) increases.
     \begin{table}
        \caption{Critical Values with different \( r \) and \( \alpha \)}
          \vspace{0.3cm}
           \centering
            \begin{tabular}{l l l l l l l l l l l}
              \hline\hline
              \( \alpha \) = 0.25,& \( r \) = 0.5&6.123& 2.333& 0.987& 3.678e-05&  0.000 \\ \hline
              \( \alpha \) = 0.5, &\( r \) = 0.5&4.616&1.578& 0.357 &2.472e-05& 0.000  \\ \hline
              \( \alpha \) = 0.6, &\( r \) = 0.5&3.203&0.679& 0.025 &0.006& 7.278e-05  \\ \hline
              \( \alpha \) = 0.25, &\( r \) = 0.75&9.127& 3.288& 1.031& 0.126& 5.675e-05    \\ \hline
              \( \alpha \) = 0.25, &\( r \) = 1& 12.75& 4.280& 1.224& 1.095e-04 &0.000 \\
              \hline \hline
            \end{tabular}
             \label{tb:tau}
      \end{table}

Table \ref{tb:level} displays critical values for
different \( \tau \), with \( \alpha = 0.25 \), \( r = 0.5 \), \( m = 10000 \) Monte
Carlo samples, a bandwidth sequence \( \mathfrak{H}_{1} = (8, 14, 19, 25, 30, 36,
41, 52)*0.001 \), and \( \ND(0,1) \) noise. Critical values behave similarly
for different \( \tau \).

     \begin{table}
         \caption{Critical Values with Different \( \tau \)}
           \vspace{0.3cm}
           \centering
             \begin{tabular}{l l l l l l l l l l}
               \hline\hline
               \( \tau \) = 0.05&6.464& 2.204& 0.620& 3.345e-05&0.000\\ \hline
               \( \tau \) = 0.5&7.997&3.089& 0.986 &0.300e-05& 0.000 \\ \hline
               \( \tau \) = 0.75&9.203&3.910& 1.106 &0.123& 7.254e-05  \\ \hline
               \( \tau \) = 0.95&8.589& 5.452& 1.904& 0.334& 1.203e-05 \\
               \hline \hline
             \end{tabular}
             \label{tb:level}
     \end{table}
Table \ref{tb:level1} displays the critical values for three alternative bandwidth
sequences:
\begin{EQA}
    \mathfrak{H}_{1}
    &=& (8, 14, 19, 25, 30, 36, 41, 52)*0.001
    \\
    \mathfrak{H}_{2}
    &=&
    (8, 16, 25, 36, 49, 63, 79, 99)*0.001
    \\
    \mathfrak{H}_{3}
    &=&
    (5, 8, 14, 19, 27, 36,46, 58)*0.001
\label{bandseqqr}
\end{EQA}
with \( \alpha = 0.25 \), \( r = 0.5 \), and \( \tau = 0.85 \).
\begin{table}[h]
      \caption{Critical Values with Different Bandwidth Sequences}
        \vspace{0.3cm}
         \centering
         \begin{tabular}{l l l l l l l l l l}
           \hline\hline
           \( \mathfrak{H}_{1} \)&11.33& 1.243& 6.933e-05& 0.000&0.000\\ \hline
           \( \mathfrak{H}_{2} \)&18.39&6.479& 2.230&0.469 &8.738e-05  \\ \hline
           \( \mathfrak{H}_3 \)&6.123& 2.333& 0.987& 3.678e-05&0.000  \\ \hline
            \hline
           \end{tabular}
            \label{tb:level1}
\end{table}
Although the critical values differ for different bandwidth sequences, \( \alpha \),
\( r \) and \( \tau \), they indicate the same patterns (finite and decreasing).

We simulate from different data generating processes, namely the distribution of
\( \varepsilon_{i} \) (given by the density \(\logdens(\cdot)\)) does not necessarily coincide with the likelihood
(\( ALD_{\tau} \)) taken to simulate critical values. Table \ref{tb:level3} presents
critical values simulated under \( t(3) \), \( \ND(0,1) \) and \( ALD_{\tau} \). The
critical values show the same trend with some differences, so we conclude that a
misidentification of error distribution would not significantly contaminate the
confidence sets.
    \begin{table}[h]
      \caption{Critical Values with Different Noise Distributions}
       \vspace{0.3cm}
       \centering
       \begin{tabular}{l l l l l l l l l }
          \hline\hline
          \( \ND(0,1) \)&11.50&4.924& 2.514 &1.313& 2.765e-05  \\ \hline
          \( ALD_\tau \)&14.05& 6.554& 3.304& 1.443& 5.879e-05 \\ \hline
          \( t(3) \)&15.42&8.707& 2.370&0.342& 3.898e-05 \\
          \hline \hline
       \end{tabular}
       \label{tb:level3}
    \end{table}

    In Table \ref{tb:level4}, critical values are shown in the same circumstances as in Table \ref{tb:level3} for the local linear case. Since introducing one more variable (trend), critical values doubled or tripled compared to the local constant case. The behavior with respect to tail functions stays the same.
   \begin{table}[h]
      \caption{Critical Values with Different Noise Distributions in Local Linear Case}
       \vspace{0.3cm}
       \centering
         \begin{tabular}{l l l l l l l r l }
           \hline\hline
           \( \ND(0,1) \)& 29.97&58.64& 43.21&33.41& 19.43 & 07.40 \\ \hline
           \( ALD(0.5) \)&45.28& 74.51& 66.43& 50.42& 31.42&13.50\\ \hline
           \( t(3) \)& 51.77&84.94& 59.28&44.99& 29.07 & 11.57 \\
           \hline \hline
         \end{tabular}
         \label{tb:level4}
    \end{table}

\subsection{Comparison of Different Bandwidth Selection Techniques}
We illustrate our proposal by considering  \( x \in [0,1] \), \( \tau= 0.75 \).
The sample with (\( n = 1000 \)) are simulated under three scenarios:
\begin{EQA}
    \fs^{[1]}(x)
    &=&
    \left\{
        \begin{array}{rl}
                0&  \mbox{if \( x \in [0,0.333] \)};\\
                8&  \mbox{if \( x \in (0.333,0666] \)};\\
               -1&  \mbox{if \( x \in (0.666,1] \)}
        \end{array}
    \right.
    \\
    \fs^{[2]}(x)
    &=&
    2x(1+x) ,
    \\
    \fs^{[3]}(x)
    &=&
    \sin(k_{1} x) + \cos(k_{2} x)\Ind\{x \in (0.333,0.666)\} + \sin(k_{2} x)
\end{EQA}
The noise distributions are: \( \ND(0,0.03), ALD_{\tau}, t(3) \).

    Figure \ref{fig:simconstant} presents pictures on comparisons of different
    estimators in the local constant case. Figure \ref{fig:lineartrend} and
    \ref{fig:map3} show in the local linear case the estimators of the functions (\( \hat{f}(x) \)) and
    its first derivatives as well. Our technique provides closer fits to the
    true curve (\( f(x) \)) than methods with a global fixed bandwidth, especially in the
    presence of jump. Table \ref{tb:level7}, which shows the averaged absolute
    errors for the four methods, further confirms our conclusion.

   \begin{table}[h]
     \caption{Comparison of Monte Carlo errors, averaged over 1000
     samples}
     \vspace{0.3cm}
       \centering
       \begin{tabular}{l l l l l}
         \hline\hline
         &Fixed bandw & Local constant & Local linear & Fixed bandw (Cai) \\
         \( \fs^{[1]}(x) \)      &0.654         & 0.172          & 0.169         &0.378\\ \hline
         \( \fs^{[2]}(x) \)      &0.206         & 0.008          & 0.008         &0.245\\ \hline
         \( \fs^{[3]}(x) \)      &0.137         & 0.021          & 0.019         &0.123 \\
         \hline \hline
       \end{tabular}
       \label{tb:level7}
    \end{table}

 Table \ref{tb:level8} offers a further analysis for misspecified error distributions. Specifically, to evaluate the accuracy of our estimation for error distributions generated differently than the ALD density. Table \ref{tb:level8} gives \( L_{1} \) errors between \( \hat{f}(\cdot)\) (with critical values simulated from \( ALD_{\tau} \)) and \( f(\cdot) \), from which we
conclude that mis-specification of error distributions would not contaminate our results
significantly.

   \begin{table}[h]
   \caption{Comparison of error mis-specification, errors are calculated averaged over \( 1000 \) samples}
   \vspace{0.3cm}
   \centering
    \begin{tabular}{l l l l }
       \hline\hline
       &Local constant \{\( \ND(0,1) \)\}& Local constant \{\( t(3) \)\}& Local linear \{\( \ND(0,1) \)\} \\
       \( \fs^{[1]}(x) \)      &0.252         & 0.220          & 0.169       \\ \hline
       \( \fs^{[2]}(x) \)      &0.070         & 0.016          & 0.043         \\ \hline
       \( \fs^{[3]}(x) \)      &0.009         & 0.021          & 0.019       \\
   \hline \hline
   \end{tabular}
    \label{tb:level8}
   \end{table}

    \begin{figure}[htbp]
      \centering
      \mbox{\subfigure[Normal]{\mygraphics{0.36}{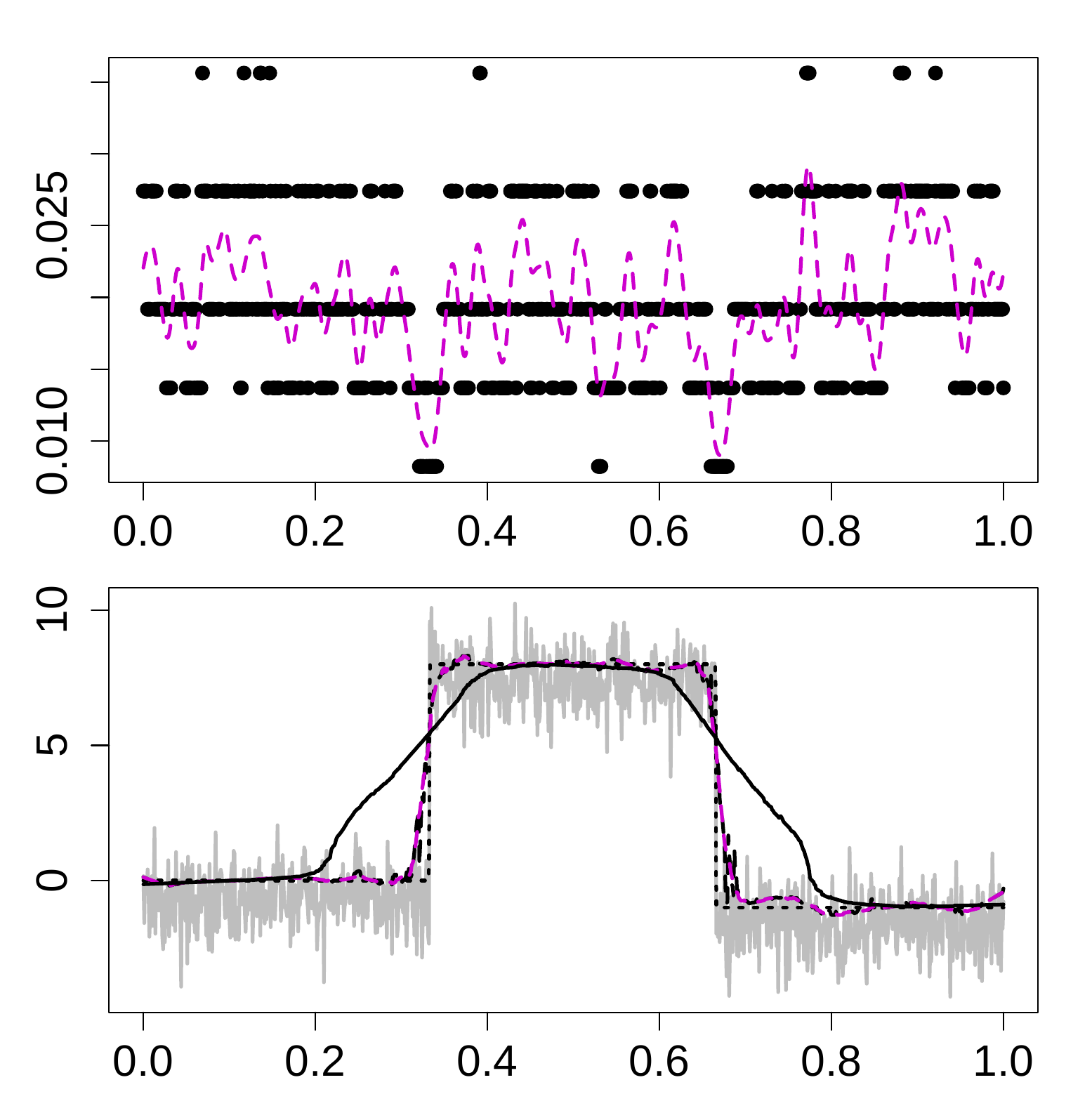}\label{fig:subfig1}}
            \subfigure[Normal]{\mygraphics{0.39}{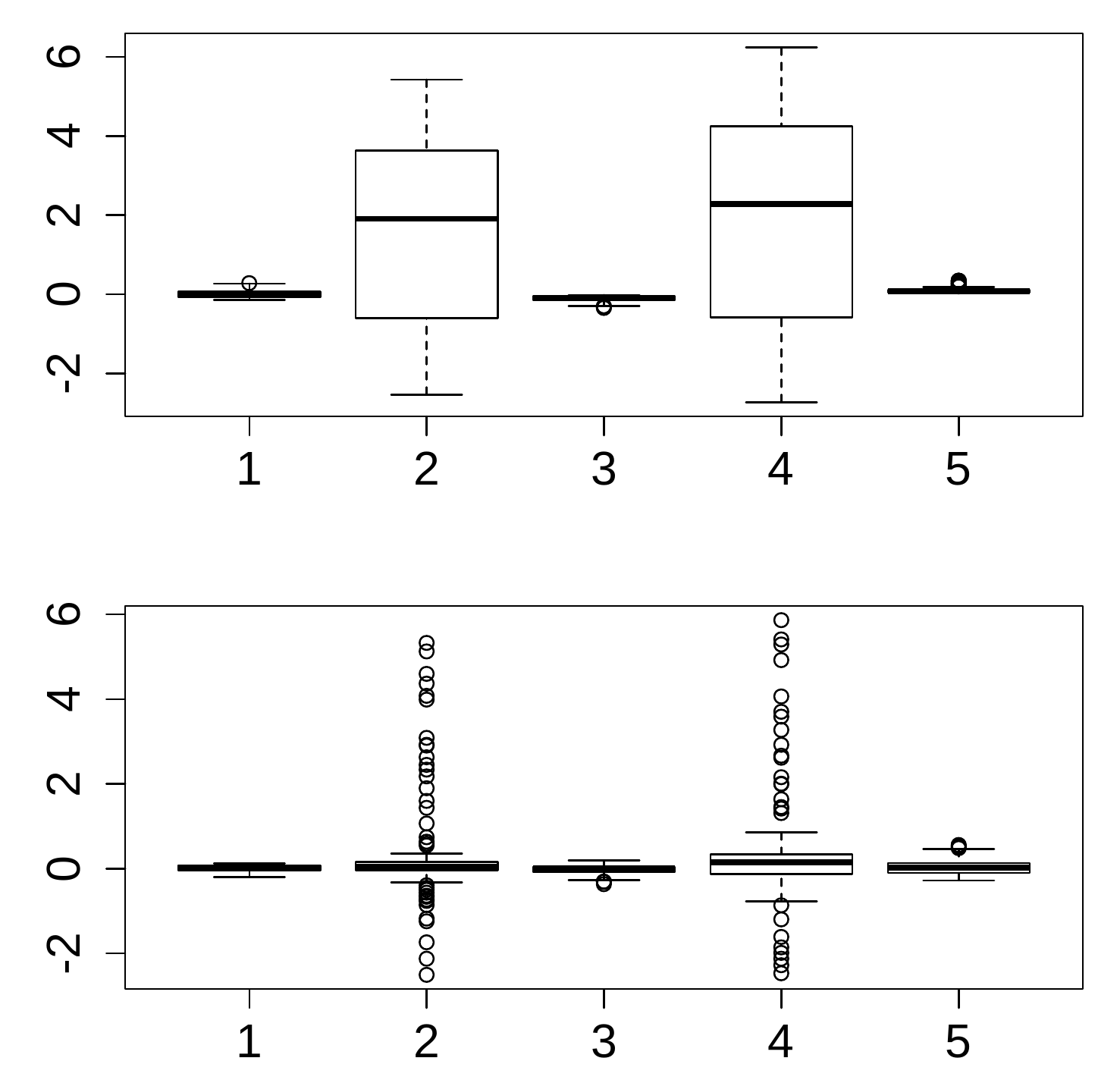}\label{fig:subfig2}}}
      \mbox{\subfigure[ALD(0.5)]{\mygraphics{0.36}{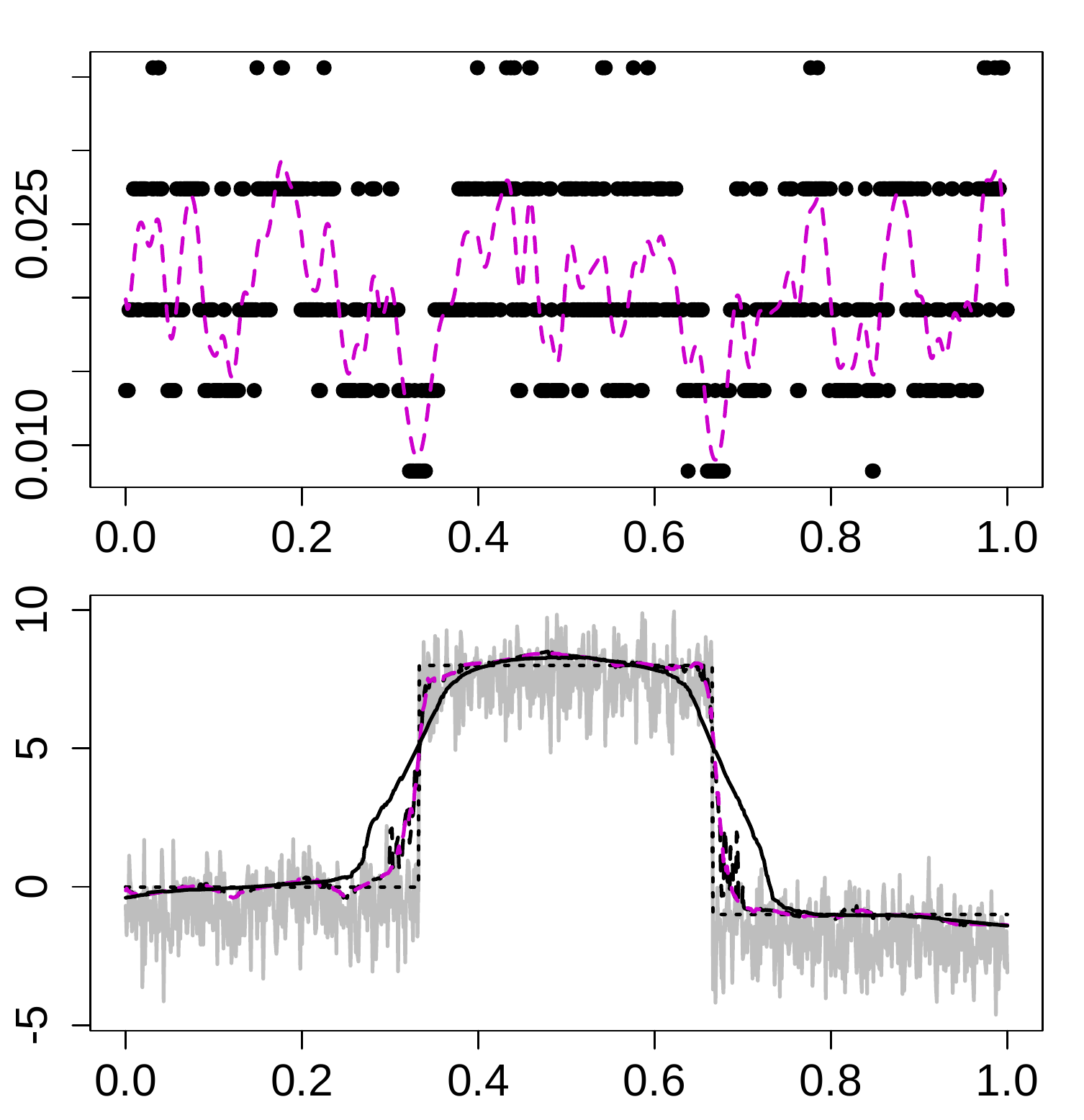}\label{fig:subfig3}}
            \subfigure[ALD(0.5)]{\mygraphics{0.39}{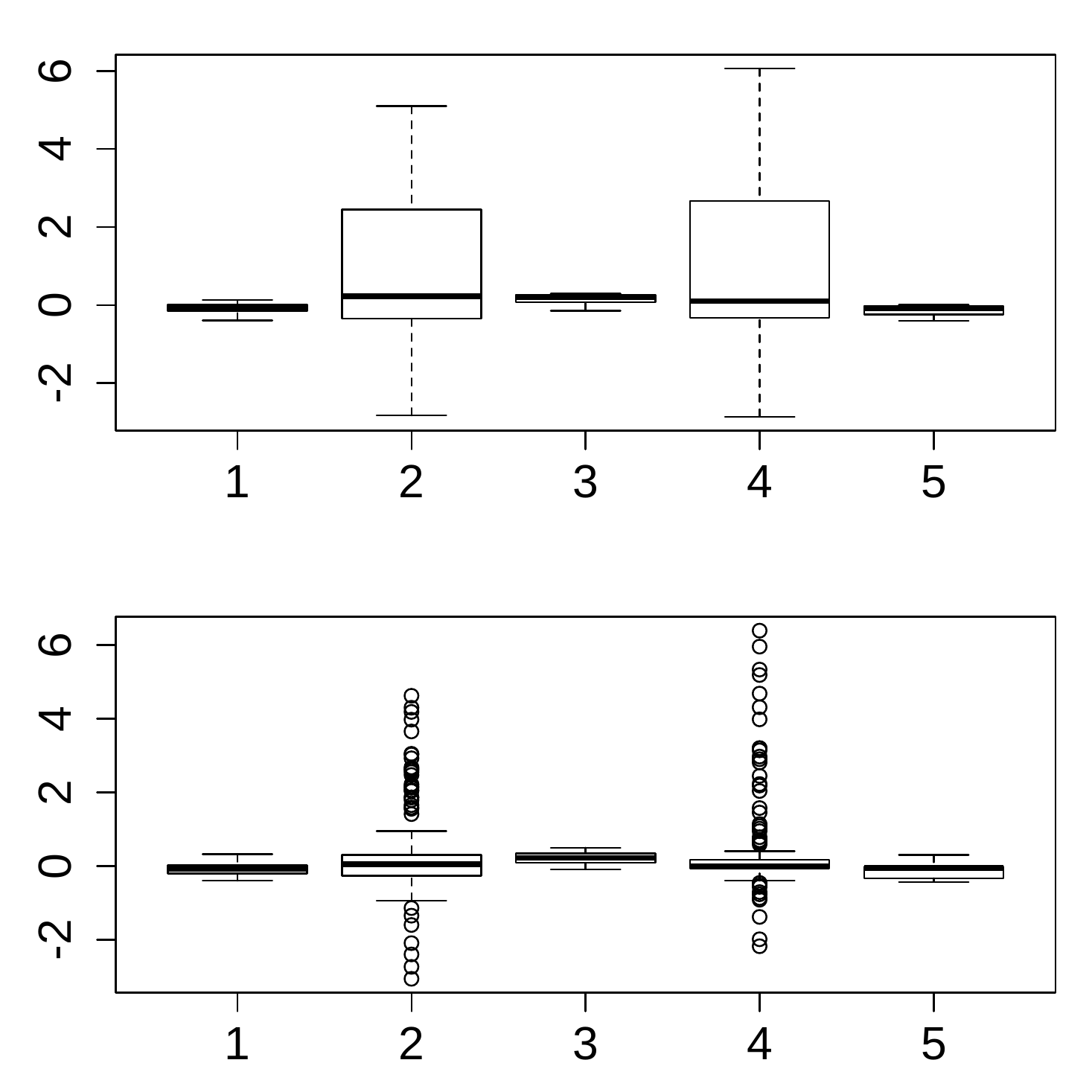}\label{fig:subfig4}}}
      \mbox{\subfigure[t(3)]{\mygraphics{0.36}{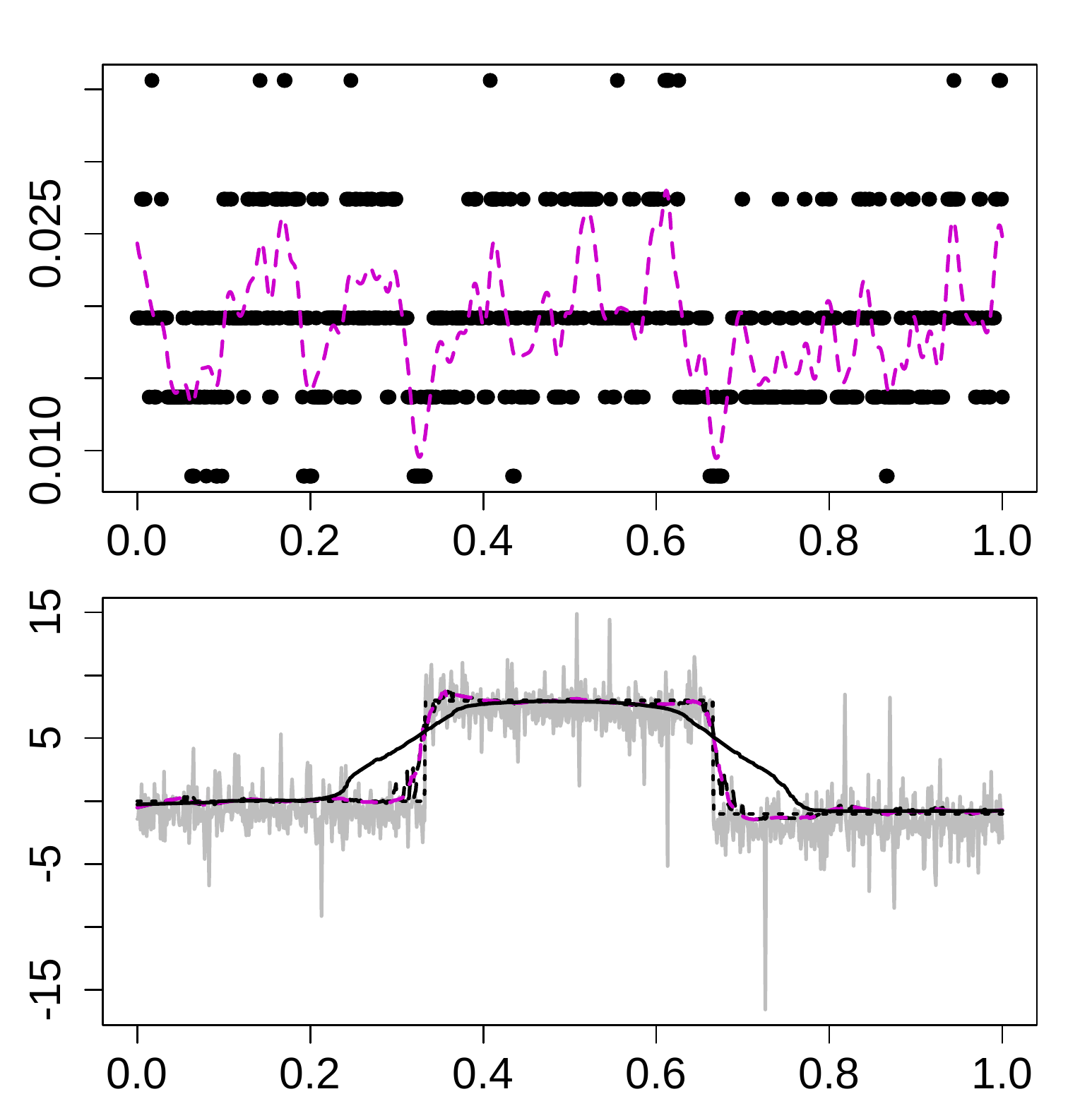}\label{fig:subfig5}}
            \subfigure[t(3)]{\mygraphics{0.36}{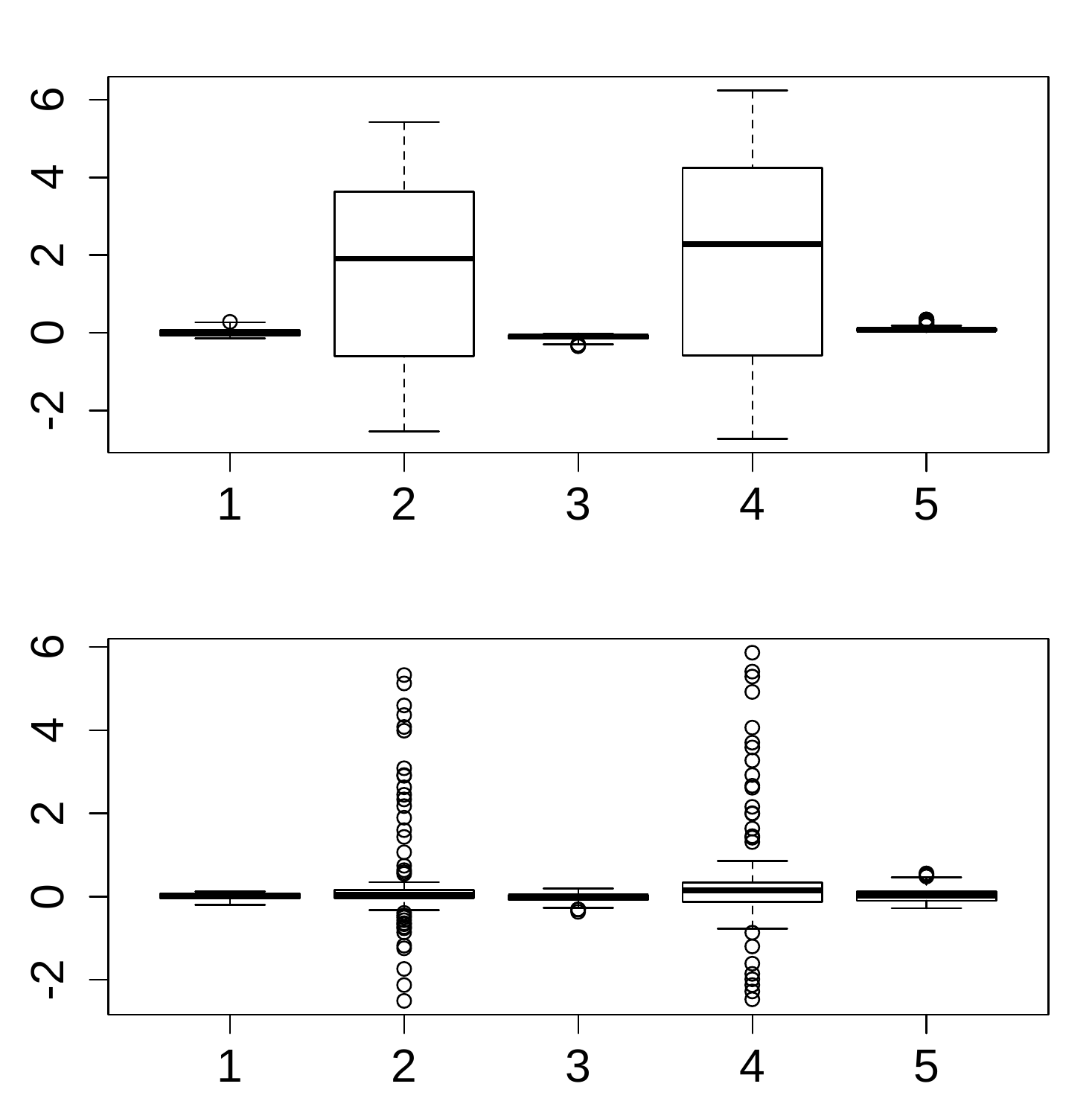}\label{fig:subfig6}}}
      \caption{The bandwidth sequence (upper left panel), the smoothed bandwidth (magenta dashed); the data with noise (grey, lower left panel), the adaptive estimation of \( 0.75 \) quantile (dashed black), the quantile smoother with fixed optimal bandwidth \( =0.06 \) (solid black), the estimation with smoothed bandwidth (dashed magenta); boxplot of block residuals fixed bandwidth (upper right), adaptive bandwidth (lower right)}\label{fig:simconstant}
    \end{figure}

    \begin{figure}[htbp]
        \centering
        \mbox{\subfigure[]{\mygraphics{0.44}{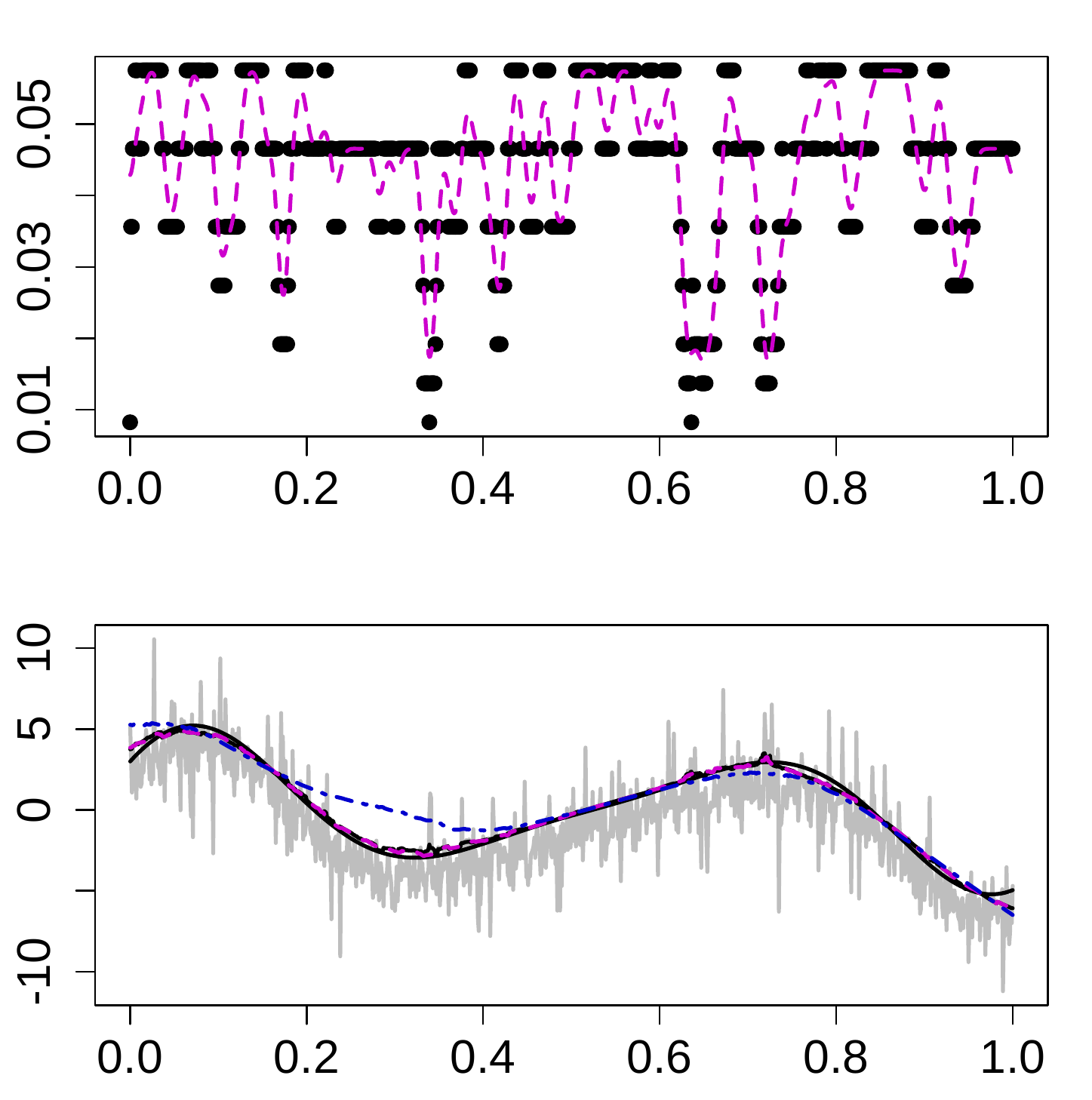}\label{fig:subfig7}}
              \subfigure[]{\mygraphics{0.44}{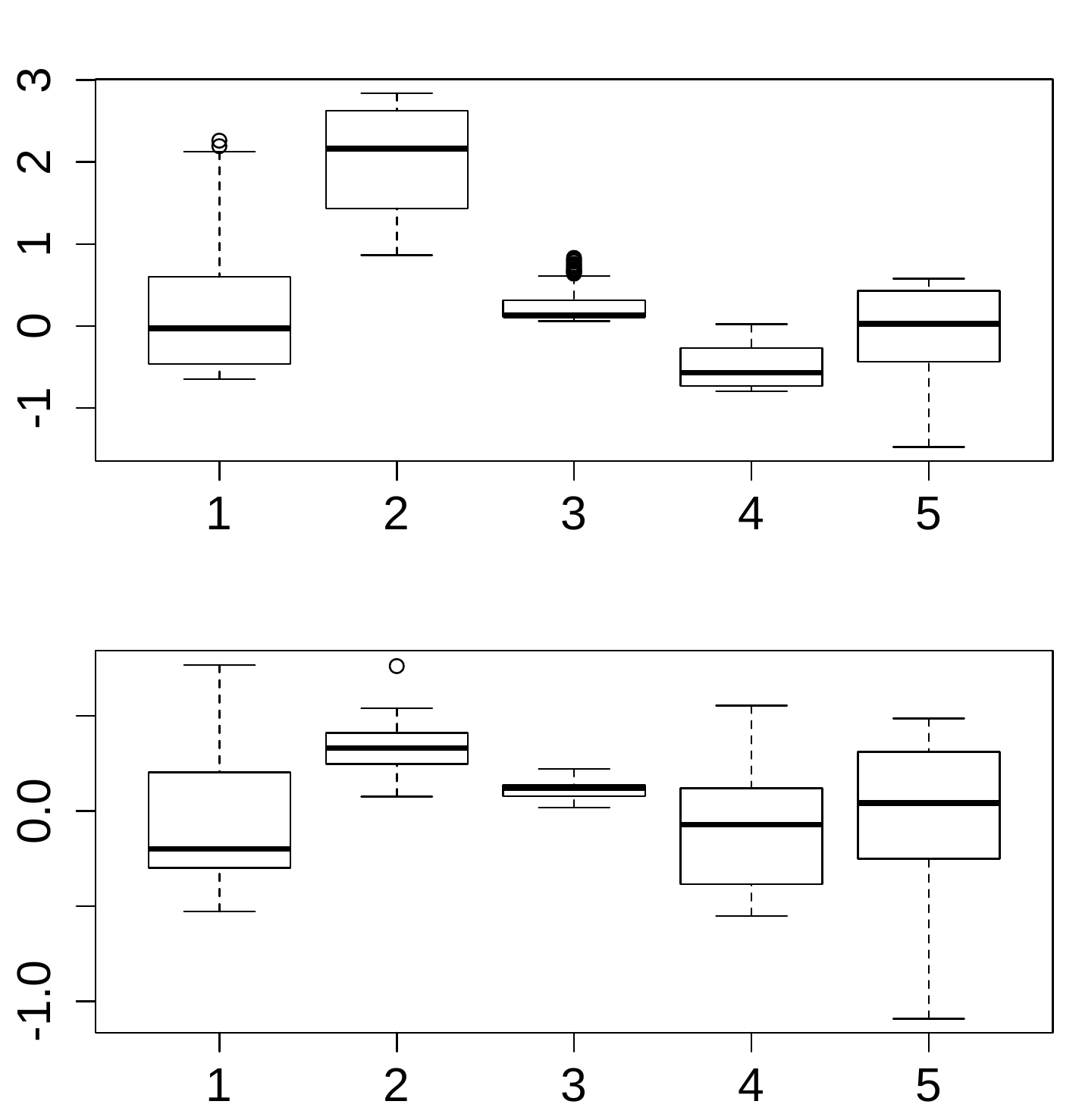}\label{fig:subfig8}}}
        \caption{ The bandwidth sequence (upper left panel), the smoothed bandwidth sequence (dashed magenta); the observations (grey, lower left panel), the adaptive estimation of \( 0.75 \) quantile (dotted black), the true curve (solid black), the quantile smoother with fixed optimal bandwidth \( =0.063 \) (dashed dotted blue), the estimation with adaptively smoothed bandwidth (dashed magenta); the blocked error of the adaptive estimator (lower right); the fixed estimator (upper right).}\label{fig:lineartrend}
    \end{figure}

     \begin{figure}[htbp]
         \centering
         \mbox{\subfigure[]{\mygraphics{0.44}{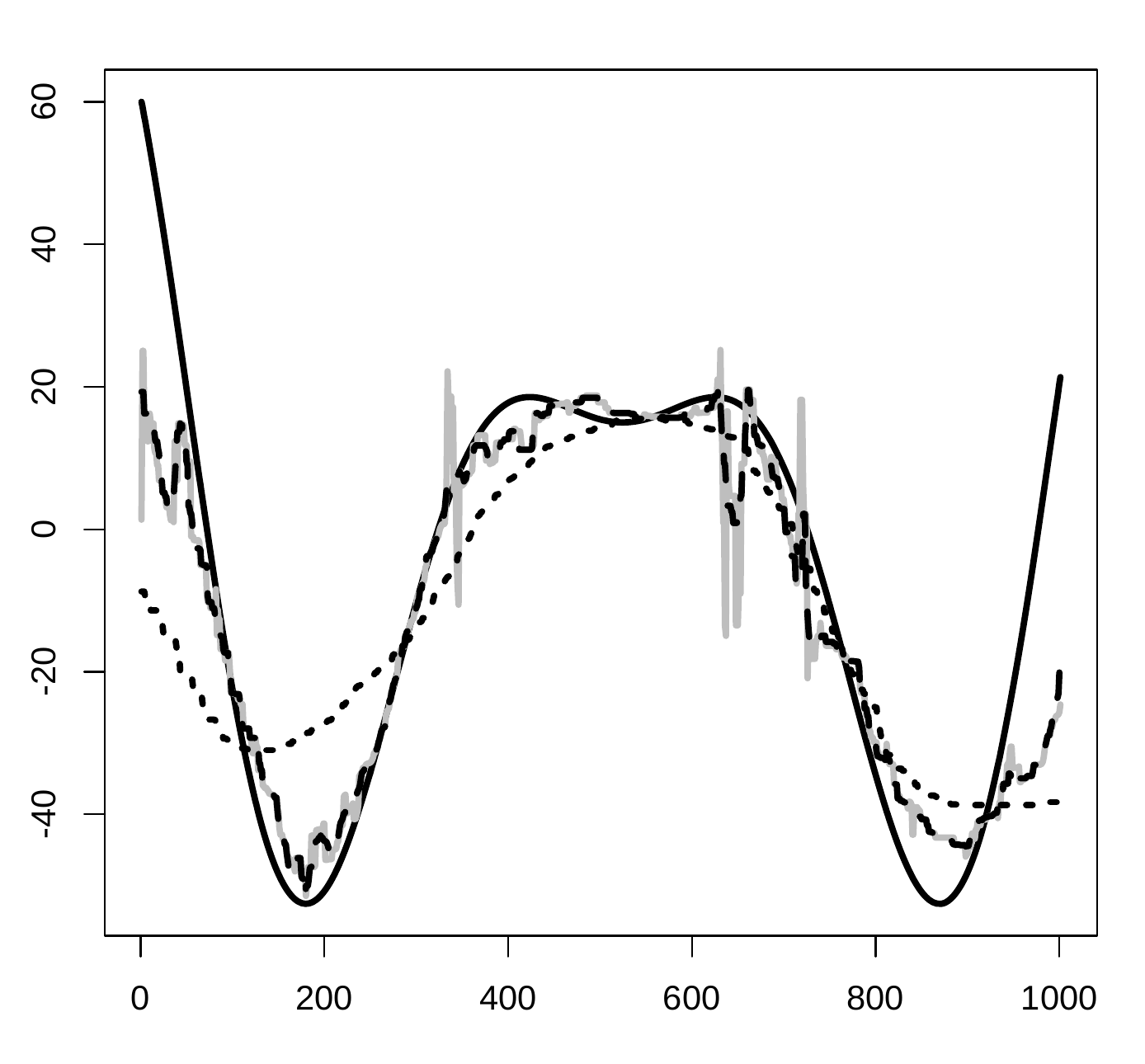}\label{fig:subfig1}}
               \subfigure[]{\mygraphics{0.38}{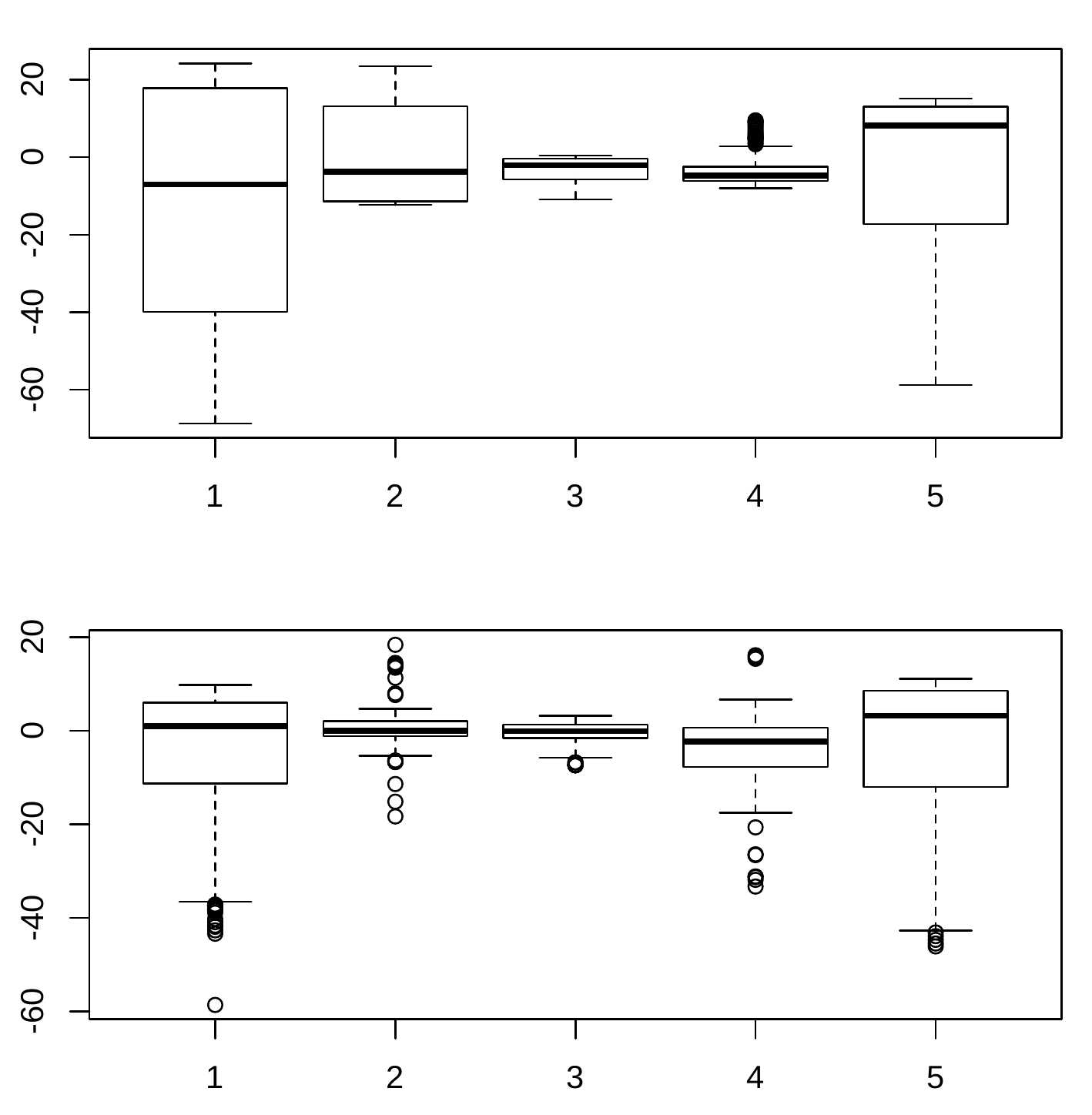}\label{fig:subfig2}}}
         \caption{The adaptive estimation of first derivative of the above quantile function (left panel grey), the true curve (solid black), the estimation with smoothed bandwidth (dashed black), the quantile smoother with fixed optimal bandwidth \( =0.045 \) (dotted black); the blocked error of the adaptive estimator (lower right); the fixed estimator (upper right).
         }\label{fig:map3}
     \end{figure}

\section{Applications}
\label{SapplLQR}
In the study of financial products, it is very important to detect and understand tail
dependence among underlyings such as stocks. In particular, the tail dependence
structure represents the degree of dependence in the corner of the lower-left quadrant
or upper-right quadrant of a bivariate distribution. \citeasnoun{ha:da:01} and
\citeasnoun{em:mc:st:99} provide good access to the literature on tail dependence and
Value at Risk. With the adaptive quantile technique, we provide an alternative approach
to studying tail dependence.

The correlation is calibrated from real data as given in Figure \ref{fig:emp2}, where \(
X \) is standardized return from stock ``clpholdings'' from Hong Kong Hangseng Index,
and \( Y \) is return from stock ``cheung kong''. The conditional quantile function is
linear, for example, \( X_{1} \in  \ND(u_{1}, \sigma_{1}) \) and \( X_{2} \in
\ND(u_{2}, \sigma_{2}) \), the conditional quantile function \( \alpha \) is:
\begin{EQA}[c]
    \fs(x)
    =
    \varphi^{-1}(\alpha)(\sigma_{2}- \sigma_{12}^{2}/\sigma_{1}) + u_{i}
    + \sigma_{12}\sigma_{2}^{-1}(x - u_{2}).
\end{EQA}
Figure \ref{fig:emp2}  show the empirical conditional quantile
curves actually deviate from the one calculated from normal distributions, which implies
non normality. The motivation of adaptive bandwidth selection is clear to see from
Figure \ref{fig:emp2}, the dependency structure change is more
obvious compared with the fixed bandwidth curve. Moreover, the flexible adaptive curve
is not likely to be a consequence of overfitting since it mostly lies in the confidence
bands produced by fixed bandwidth estimation, see \citeasnoun{wo:so:10}.
%

We measure the deviation from normality by accumulated \( L_{1} \) distance  to  the normal
fitting and examine different combination of stocks from Hong Kong Hangseng Index. The
results is summarized in Table \ref{tb:stock8}.

%

     \begin{figure}
         \begin{center}
         \mygraphics{0.5}{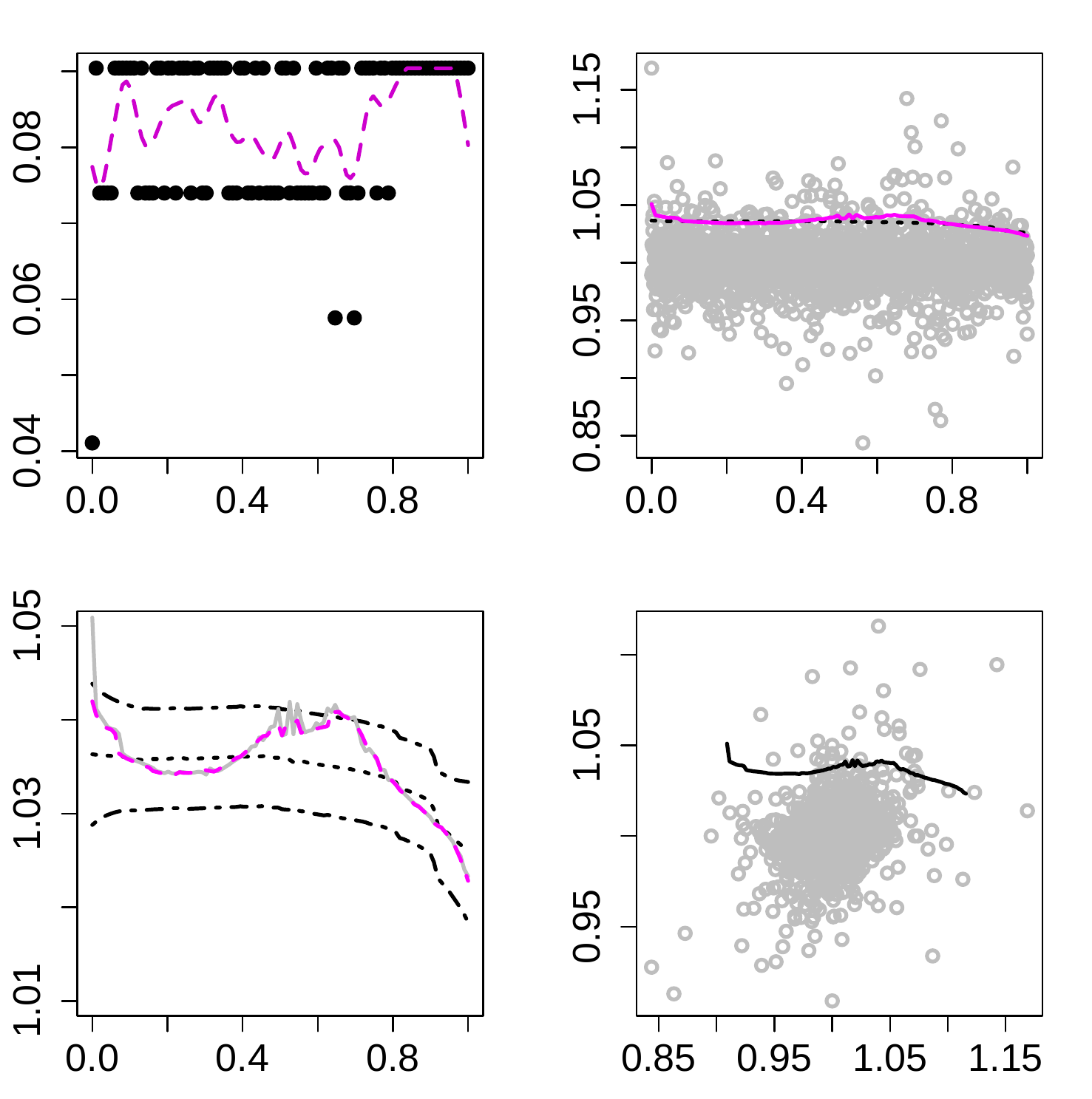}\\
           \caption{The bandwidth sequence with smoothed bandwidth curve(upper left
           panel), the smoothed bandwidth (dashed magenta); Scatter plot of stock
           returns  (upper right panel), the adaptive estimation of \( 0.90 \) quantile
           (solid magenta), the quantile smoother with fixed optimal bandwidth \( =0.15
           \) (dotted black); fixed bandwidth curve (dotted black), adaptive bandwidth
           curve (grey), the estimation with smoothed bandwidth (dashed magenta),
           confidence band (dashed black) (lower left panel);  adaptive bandwidth with
           normal scale (lower right). }
           \label{fig:emp2}
          \end{center}
     \end{figure}
%
%
%


    \begin{table}[h]
      \caption{Summary of deviation from normality }
      \vspace{0.3cm}
        \centering
        \begin{tabular}{l l l l l}
              \hline\hline
                            &Chalco        & Cosco pacific  & Bank of China  \\
       New world devo       &\( 0.252 \)         & \( 0.220 \)          & \( 0.169 \)       \\ \hline
       Sino land            &\( 0.070 \)         & \( 0.016 \)          & \( 0.043 \)         \\ \hline
       Swire pacific A      &\( 0.009 \)         & \( 0.021 \)          & \( 0.019 \)       \\
             \hline \hline
        \end{tabular}
        \label{tb:stock8}
     \end{table}

Another application of quantile function estimation is in temperature data analysis,
which is of key interest for  pricing temperature derivatives. Quantile regression can
provide a more flexible and comprehensive approach to understand the temperature risk
drivers defined in (\refeq{eq:no2}).

Denote daily temperature as \( T\mapsto(t,j) \), with \( t=1,\cdots ,\tau=365 \) days,
\( j=0,\cdots, J \) years. The time series decomposition for \( T_{t,j} \) is given as:
\begin{EQA}
        X_{t,j}
        &=&
        T_{t,j}-\Lambda_{t}
        \\
        X_{t,j}
        &=&
        \sum^{L}_{l=1} \beta_{l}X_{t-l,j} + \sigma_{t} \eta_{t,j}
        \\
        \eta_{t,j}
        & \sim &
        \ND(0,1),
        \\
        \varepsilon_{t,j}
        & \eqdef &
        \sigma_{t}\varepsilon_{t,j}
        \\
        \hat{\varepsilon}_{t,j}
        & \eqdef &
        X_{365j+t}- \sum^{L}_{l=1} \hat{\beta}_{l} X_{365j+t-l}
        \label{eq:no2}
\end{EQA}
where \( T_{t,j} \) is the temperature at day \( t \) in year \( j \), \( \Lambda_{t} \)
denotes the seasonality effect and \( \sigma_t \) the seasonal volatility.

We are interested specifically in the stochastic risk drivers \( \varepsilon_{t,j} \),
Figure \ref{fig:weatherres} presents a time series plot of \(
\hat{\varepsilon}_{t,j}/\hat{\sigma}_t \), and the estimated \( 90\% \) quantile
function. By zooming in the curve, we observe a very interesting phenomena: a changing
of the trend of the standardized residual over years.

     \begin{figure}
      \centering
      \mygraphics{0.5}{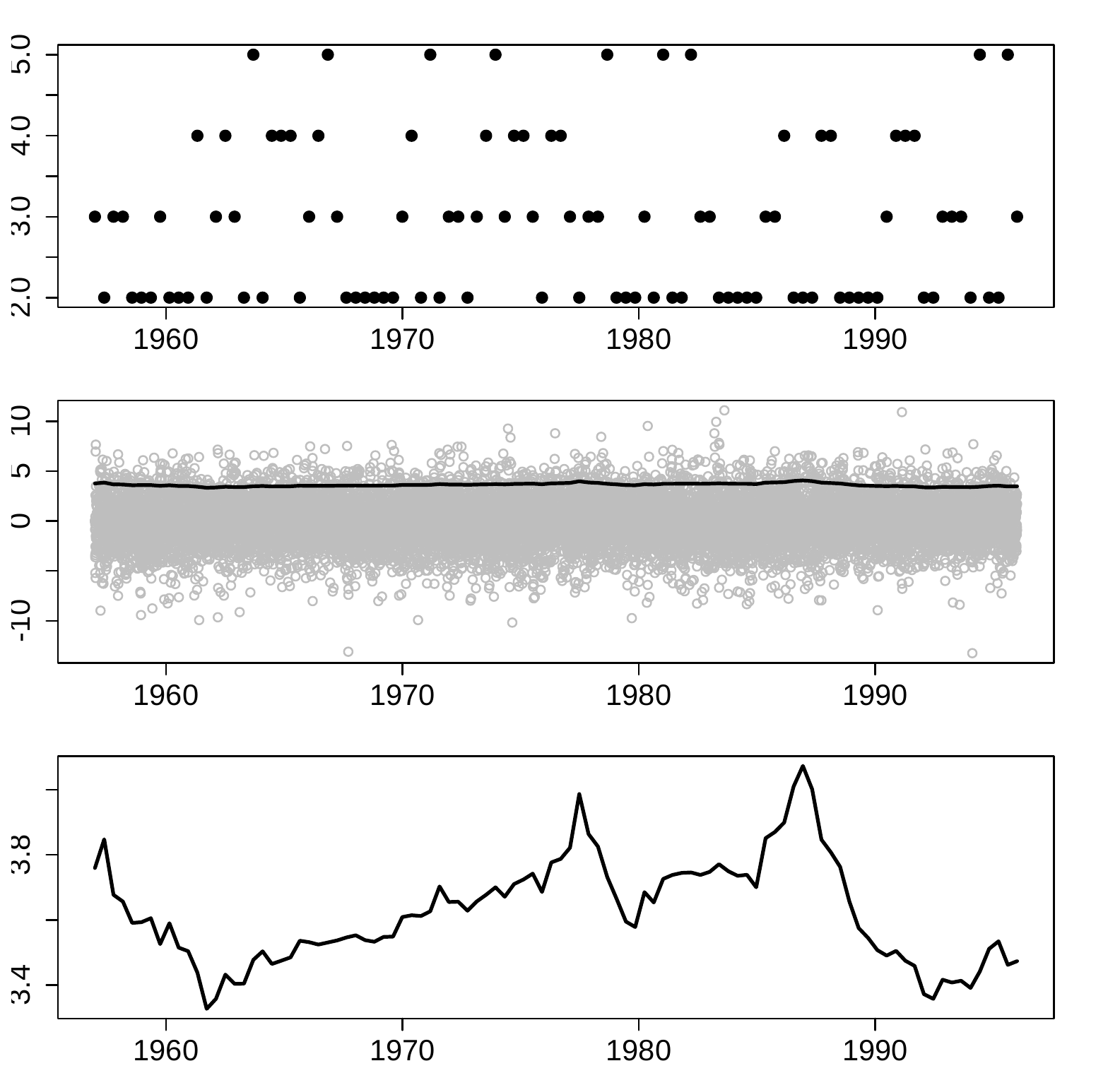}
       \caption{Plot of quantile curve for standardized weather residuals over \( 40 \)
       years in Berlin, \( 95\% \)  quantile, \( 1967-2006 \). Selected bandwidths
       (upper), observations with estimated the quantile function (middle), the
       estimated the quantile function (lower).}
       \label{fig:weatherres}
     \end{figure}

To further understand the risk factors, we analyze the quantile functions of \(
\hat{\varepsilon}_{t,j}^{2} \) over \( 12 \) years, and average over  \( 4 \) years  for
comparison, see Figure \ref{fig:09quantile} and Figure \ref{fig:10quantile}. The
differences between Berlin and Kaoshiung are easy to see, the variance function has a high
value for Jan-Feb, while for Berlin the peaks  and to come more in summer. Moreover, there is a
tendency for Kaoshiung to be more volatile over time, but this phenomenon does not
appear in Berlin.

     \begin{figure}
        \begin{center}
        \mygraphics{0.5}{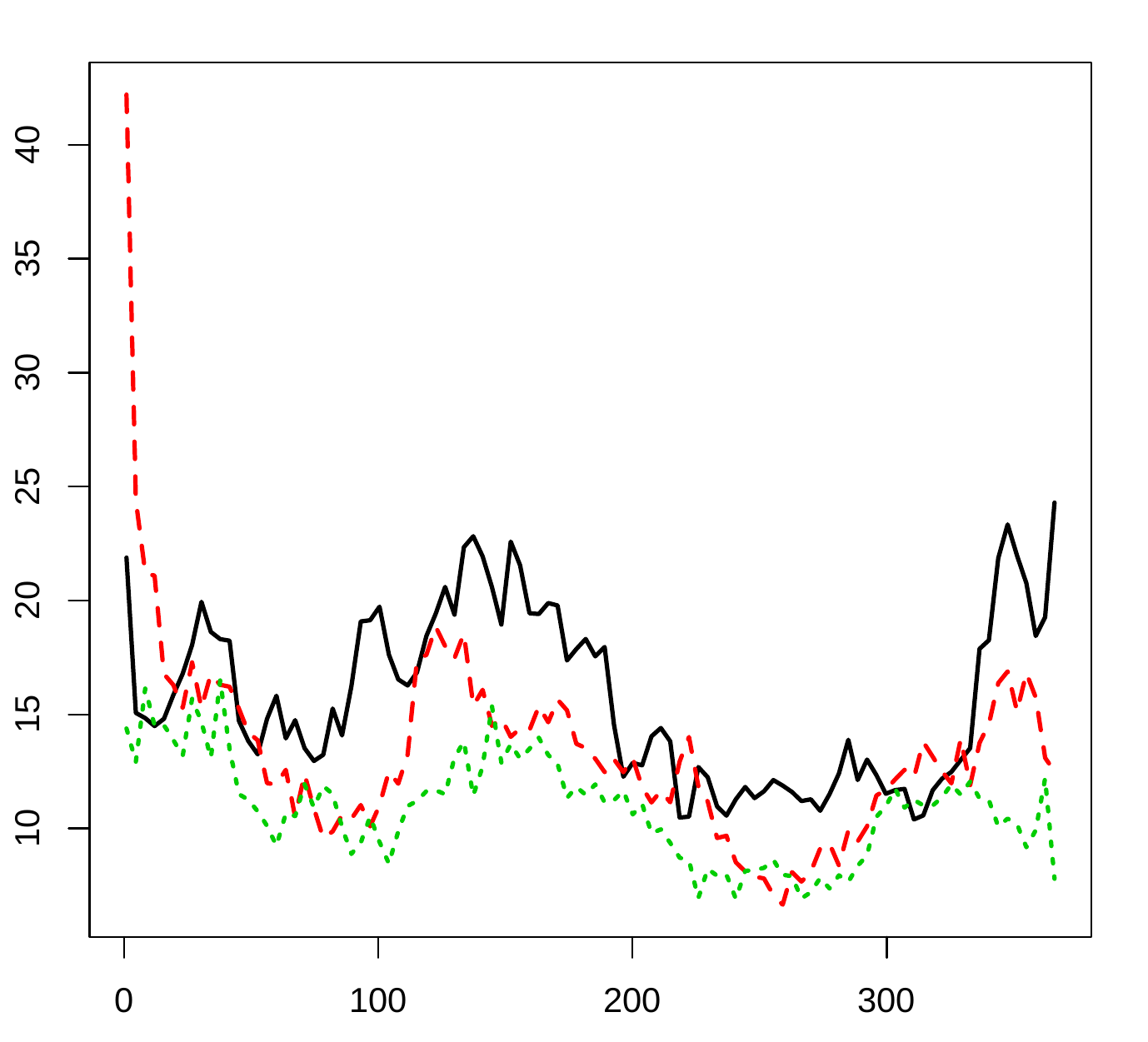}\\
        \caption{Estimated \( 90\% \) quantile of variance functions, Berlin, average
        over \( 1995-1998 \),  \( 1999-2002 \) (red), \( 2003-2006 \)
        (green)}\label{fig:09quantile}
        \end{center}
     \end{figure}

     \begin{figure}
        \begin{center}
        \mygraphics{0.5}{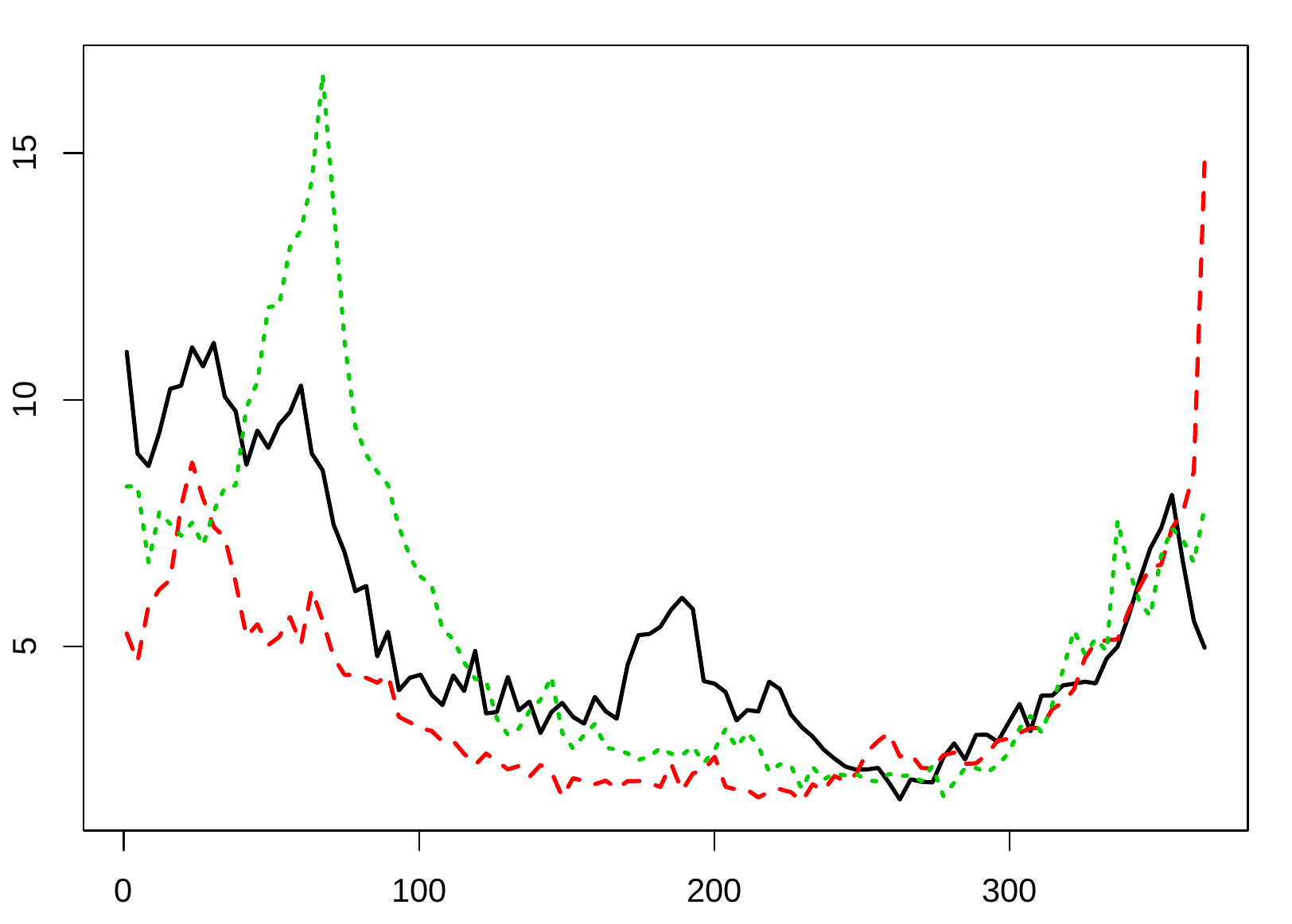}\\
        \caption{Estimated \( 90\% \) quantile of variance functions, Kaoshiung, average
        over \( 1995-1998 \),  \( 1999-2002 \) (red), \( 2003-2006 \)
        (green)}\label{fig:10quantile}
       \end{center}
     \end{figure}


In addition, our technique can also be used to estimate the function \( \sigma_t \).
We propose four methods: 1, Estimate the median curve of \( \hat{\varepsilon}_{t,j} \)
using adaptive technique. 2, Take  \( \{\hat{f}_{\varepsilon,0.75}-
\hat{f}_{\varepsilon,0.25}\}/1.34 \) (\( 1.34 \) is the inter quartile range of a
standard normal distribution), where \( \hat{f}_{\varepsilon,0.75} \), \(
\hat{f}_{\varepsilon,0.25} \) are the adaptive quantile estimators. 3, Estimate the mean curve of
\( \hat{\varepsilon}_{t,j} \) using adaptive bandwidth. 4, Estimate the mean function of
\( \hat{\varepsilon}_{t,j} \) with a fixed bandwidth. The aforementioned methods are
compared by testing the normality of \( \hat{\eta}_{t,j} =
\hat{\varepsilon}_{t,j}/\hat{\sigma}_t  \). As according to our normal assumption on \(
\eta_{t,j} \), a good estimation for \( \sigma_t \) leads to normal standardized
residuals \( \hat{\eta}_{t,j} \). Table \ref{tb:weather7} and  \ref{tb:weather8}
summarize statistics from the normality test of standardized residuals from three
methods  in Berlin and Kaoshiung. It can be seen that Berlin has more normal residuals
than Kaoshiung. Method three is always better at getting more normal residuals, and
method two is compatible with method three. It may be due to the fact that quantiles at higher or
lower levels are better at explaining the extremes of the volatility function. Method
four  performs not so well,  as it is with a fixed bandwidth. Therefore we conclude that
our adaptive technique is useful in modeling temperature residuals.

    \begin{table}[h]
    \caption{P-values of Normality Tests: Berlin }
       \vspace{0.3cm}
       \centering
       \begin{tabular}{l l l l l}
       \hline\hline
       &AD& JB& KS \\
       \( 1 \)      &0.000         & 0.010          & 0.060       \\ \hline
       \( 2 \)      &0.062         & 0.000          & 0.020        \\ \hline
       \( 3 \)      &0.054         & 0.487          & 0.171       \\   \hline
       \( 4 \)      &0.009         & 0.000          & 0.002       \\
        \hline \hline
       \end{tabular}
     \label{tb:weather7}
    \end{table}

    \begin{table}[h]
      \caption{P-values of Normality Tests:Kaoshiung }
      \vspace{0.3cm}
      \centering
     \begin{tabular}{l l l l l}
       \hline\hline
       &AD& JB& KS \\
       \( 1 \)   & 0.000 &0.000 &0.000      \\ \hline
       \( 2 \)      & 1.03e-05& 0.077& 0.043     \\ \hline
       \( 3 \)      &2.37e-06     &0.742       &0.674      \\   \hline
       \( 4 \)      &0.000         & 0.021          & 0.019       \\
      \hline \hline
     \end{tabular}
     \label{tb:weather8}
    \end{table}

\section{Finite Sample Theory}
\label{StheoryLQR}
This section discusses some theoretical properties of the proposed estimator
\( \hat{\thetav}(x) = \tilde{\thetav}_{\hat{k}}(x) \)  under a general data distribution.
Here \( \hat{k} = \hat{k}(x) \) is the index selected by the pointwise procedure
from Section~\ref{Spointband}.
The main ``oracle'' result shows that \( \hat{\thetav}(x) \) is \emph{adaptive} in
the sense that it provides nearly the same quality of estimation as the \emph{oracle}
estimator \( \tilde{\thetav}_{\ks}(x) \) which is the best in the family
\( \bigl\{ \tilde{\thetav}_{k}(x) \bigr\}_{k=1}^{\K} \).
A precise definition of \( \ks \) will be given below in term of the
\emph{modeling bias}.

\subsection{Modeling Bias}
\label{SModBi}
The proposed approach for the bandwidth selection suggests taking a larger and
larger bandwidth until the linear parametric assumption is not significantly violated
on the considered interval.
The likelihood ratio test statistics
\( L\bigl( W^{(\ell)},\tilde{\thetav}_{\ell}(x),\tilde{\thetav}_{k}(x) \bigr) \) from
\eqref{eq:confs} are used for this check.
The formal definition of the best or oracle choice requires the introduction of a measure
for the deviation of the function \( \fs(\cdot) \) from its best linear
approximation of the form \( \Psi^{\T} \thetav \) on the interval of radius
\( h_{k} \) considered at step \( k \) of the procedure.
We follow \cite{Spokoiny:2009} who introduced the \emph{modeling bias} to measure the
deviation from the linear parametric structure.
Define \( P_{i} \) as the distribution of the observation \( Y_{i} \).
Let also \( P_{i,s} \) be a shift of \( P_{i} \) by \( s \), that is, the
distribution of \( Y_{i} - s \).
Also denote \( \fs_{i} = \fs(X_{i}) \) and
\( \fs_{i}(\thetav) = \Psi_{i}^{\T} \thetav \).
In particular, \( P_{i,\fs_{i}} \) is the distribution of
\( \varepsilon_{i} \eqdef Y_{i} - \fs(X_{i}) \), so that its \( \tau \)-quantile is
zero.
The underlying measure \( \P \) is the product of the measures \( P_{i,\fs_{i}} \).
Under the linear PA \( \fs(X_{i}) = \fs_{\thetav}(X_{i}) \), the corresponding
measure \( \P_{\thetav} \) is the product of the \( P_{i,f_{i}(\thetav)} \):
\begin{EQA}[c]
    \P
    =
    \prod_{i=1}^{n} P_{i,f_{i}} \, ,
    \qquad
    \P_{\thetav}
    =
    \prod_{i=1}^{n} P_{i,f_{i}(\thetav)} \, .
\label{Pti1nPi}
\end{EQA}
The modeling bias at step \( k \) measures the
deviation of the true quantile function \( \fs \) from the linear parametric one and
it is defined as
\begin{EQA}
    \Delta_{k}
    & \eqdef &
    \inf_{\thetav} \Delta_{k}(\thetav),
    \\
    \Delta_{k}(\thetav)
    & \eqdef &
    \sum_{i=1}^{n} \kullb\bigl( P_{i,\fs_{i}}, P_{i,\fs_{i}(\thetav)} \bigr)
        \Ind\{ w^{(k)}_{i} > 0 \}.
\label{Deltak}
\end{EQA}
Here \( \kullb(P,Q) \) is the Kullback-Leibler divergence between two measures
\( P \) and \( Q \).
The quantity \( \Delta_{k}(\thetav) \) can be viewed as weighted
Kullback-Leibler divergence between \( \P \) and \( \P_{\thetav} \)
localized to the observations in the interval of radius \( h_{k} \) around \( x \).
The value \( \Delta_{k} \) describes the quality of the best linear approximation on
this interval.
The \emph{small modeling bias} (SMB) condition manifests that the value
\( \Delta_{k} \) does not exceed a prescribed quantity \( \Delta > 0 \),
and the oracle choice of the bandwidth \( h_{k} \) is
defined as the largest bandwidth among \( h_{k} \)
for which the SMB condition is satisfied:
\begin{EQA}[c]
\label{ksSMBDelta}
     \ks
     \eqdef
     \argmax_{k \le \K} \{ \Delta_{k} \le \Delta \} .
\end{EQA}
\cite{Spokoiny:2009} argued that such a choice leads to the bias-variance trade-off
in the usual nonparametric sense.
Thus, the oracle bandwidth yields the rate optimal estimation quality in the asymptotic
set-up.

All the introduced quantities depend on the central point \( x \).
Therefore, the parameter \( \thetavs \) of the best parametric fit and the oracle
bandwidth \( \ks \) also depend on \( x \):
our approach allows us to specify the best bandwidth for each point separately.
Under the measure \( \P_{\thetavs} \),
the estimate \( \tilde{\thetav}_{k}(x) \) is close to \( \thetavs \) in the sense that the
confidence set \( \CS_{k}(\zz_{k}) \) covers \( \thetavs \) with a high probability
and the risk
\( \E_{\thetavs} L^{r}\bigl( W^{(k)},\tilde{\thetav}_{k}(x),\thetavs \bigr) \) remains
bounded by a fixed constant \( \riskt_{r} \) for all \( k \le \K \).
The definition of the modeling bias based on the Kullback-Leibler divergence
allows for the translation of these properties to the general case at the cost of the additional
factor \( \ex^{\Delta} \).
More precisely, the following bound holds.

\begin{theorem}
\label{Toracle1}
Let \( \thetavs \) and \( \ks \le \K \) be such that
\( \Delta_{\ks}(\thetavs) \le \Delta \).
Then for each \( k \le \ks \)
\begin{EQA}
\label{eq:ora1}
    \E\log\biggl\{
        1+\frac{L^{r}(W^{(k)},\tilde{\thetav}_{k}(x),\thetavs)}{\riskt_{r}}
    \biggr\}
    & \le &
    \Delta+1 .
\end{EQA}
\end{theorem}
So, if \( \Delta \) is small all the confidence or risk bounds continue to apply
even in the local nonparametric situation.

\subsection{``Oracle'' Property}
This section presents our main result called the oracle risk bound.
The main message of this result is that the adaptive estimator
\( \hat{\thetav}(x) \) performs nearly as well as the best (oracle) estimator does.
Our theoretic study is performed under the assumption that the critical values
\( \zz_{k} \) are computed under the measure \( \P_{\thetavs} \) described in
Section~\ref{SModBi}.
Due to Lemma~\ref{Lpivotq}, a particular choice of the parameter \( \thetavs \) does not
matter.
In addition, \( \P_{\thetavs} \) involves
the distribution of the residuals \( \varepsilon_{i} - \fs(X_{i}) \) which is not
available.
However, one can use a proxy for this distribution, because the critical values are
rather stable w.r.t. to the error distribution; see Corollary~\ref{CTwilksqr} and
discussion afterwards for more arguments.

Let the bandwidth index \( \ks \) be defined by the SMB condition
\eqref{ksSMBDelta} leading to the oracle estimator \( \tilde{\thetav}_{\ks}(x) \).
The next result claims that for the final estimator \( \hat{\thetav}(x) \), the
difference
\begin{EQA}[c]
    L\bigl( W^{(\ks)},\tilde{\thetav}_{\ks}(x),\hat{\thetav}(x) \bigr)
    =
    L\bigl( W^{(\ks)},\tilde{\thetav}_{\ks}(x) \bigr)
    -L \bigl( W^{(\ks)},\hat{\thetav}(x) \bigr)
\label{LLLWksht}
\end{EQA}
is not larger in order than \( \zz_{\ks} \, \ex^{\Delta} \).
Later we show that the critical value \( \zz_{\ks} \) is at most logarithmic in the
sample size \( n \).
Therefore, the oracle result means that the adaptive estimator \( \hat{\thetav}(x)
\) belongs with a dominating probability to a confidence set of the oracle.

\begin{theorem}
\label{Toracle}
Suppose \Ass.1--\Ass.5.
Let \( \ks \le \K \) be such that
\( \Delta_{\ks}(\thetav) \le \Delta \).
Then
\begin{EQA}
\label{eq:oral2}
    \E\log\biggl\{
        1 +
        \frac{L^{r}\bigl( W^{(\ks)},\tilde{\thetav}_{\ks}(x),\hat{\thetav}(x) \bigr)}
             {\riskt_{r}}
    \biggr\}
    & \le &
    \alpha + \Delta + \log\bigl( 1+ \frac{\zz_{\ks}}{\riskt_{r}} \bigr) .
\end{EQA}
\end{theorem}

An interesting special case of this result is the pure parametric situation with a
linear (in \( \Psi \)) quantile function \( \fs \).
The oracle estimator \( \tilde{\thetav}_{\ks} \) corresponds to \( \ks = \K \), that
is, to the largest bandwidth \( h_{\K} \).
If it is large enough, then \( \tilde{\thetav}_{\K} \) nearly coincides with the
global quantile estimator.
Moreover, if the errors \( Y_{i} - \fs(X_{i}) \) are i.i.d. Laplacian, then
\( \tilde{\thetav} \) is nearly efficient.
The critical values \( \zz_{k} \) decreases with \( k \) and the largest one
\( \zz_{\K} \) is usually close to zero.
So, our oracle result yields that the proposed adaptive procedure is nearly
efficient in the parametric situation.

\section{Appendix}
The appendix collects the conditions, technical results, and the proofs.
First we fix our assumptions.
%
We assume independent observations \( Y_{1},\ldots,Y_{n} \).
The results are stated for a deterministic design \( X_{1},\ldots,X_{n} \)
under mild regularity conditions.
The case of a random design can be considered by the usual conditioning argument.
Given \( \tau \), the quantile function \( \fs(\cdot) \) is defined by the relation
\( \P\bigl\{ Y_{i} > \fs(X_{i}) \bigr\} = \tau \).
To avoid ambiguous notation, we suppose that this equation has an unique solution for
each \( i \).
The general case can be easily reduced to this one by standard arguments;
see e.g. \cite{Ro:2005}.
We also denote by \( P_{i} \) the distribution of the residual
\( \varepsilon_{i} = Y_{i} - \fs(X_{i}) \) and by \( \logdens_{i}(\cdot) \) its density.
Below a point \( x \) is fixed and the target of estimation is the quantile
\( \fs(x) \).
The local parametric approach requires fixing a localizing weighting scheme
\( W = (w_{1},\ldots,w_{n}) \) and linear parametric family
\( \fs_{\thetav}(\cdot) \) with \( \fs_{\thetav}(X_{i}) = \Psi_{i}^{\T} \thetav \),
where \( \Psi_{i,m} = (X_{i} - x)^{m}/m!\) for \( m=0,1,\ldots,\dimp \).

Our theoretical study can be separated into two parts.
An essential and the most difficult part is done under the linear
parametric assumption \( \fs(\cdot) \equiv \fs_{\thetavs}(\cdot) \),
then we extend the results to the case when this assumption is approximately
fulfilled in a local vicinity of the central point \( x \).

Below a family of localizing weighting schemes
\( W^{(k)} = \bigl\{ w_{i}^{(k)} \bigr\}_{i=1}^{n} \) for \( k = 1,\ldots,\K \)
is supposed to be fixed.
Our standard proposal is \( w_{i}^{(k)} = K_{\loc}\bigl\{ (X_{i} - x)/h_{k} \bigr\} \)
for a given kernel \( K_{\loc}(\cdot) \) and a sequence of bandwidths
\( h_{1} < h_{2} < \ldots < h_{\K} \).
Define
\begin{EQA}
    \DP_{k}^{2}
    & \eqdef &
    \sum_{i=1}^{n} \Psi_{i} \Psi_{i}^{\T} \logdens_{i}(0) w_{i}^{(k)}
\label{DPWqr}
    \\
    \VP_{k}^{2}
    & \eqdef &
    \Var\bigl\{ \nabla L(W^{(k)},\thetavs) \bigl\}
    =
    \tau (1-\tau) \sum_{i=1}^{n} \Psi_{i} \Psi_{i}^{\T} \bigl| w_{i}^{(k)} \bigr|^{2} ,
\label{VPWqr}
    \\
    N^{-1/2}_{k}
    & \eqdef &
    \max_{i \le n} \sup_{\gammav \in \R^{\dimp+1}}
        \frac{|\gammav^{\T} \Psi_{i}| \Ind\bigl( w_{i}^{(k)} > 0 \bigr)}{ \| \VP_{k} \gammav \| }
        \sqrt{\tau (1-\tau)} .
\label{NWqr}
\end{EQA}
Here \( \DP_{k}^{2} \) and \( \VP_{k}^{2} \) are symmetric \( (\dimp+1) \times (\dimp+1) \)
matrices: \( \DP_{k}^{2} \) can be defined similarly to the Fisher information
matrix \( \DP_{k}^{2} = - \nabla^{2} \E L(W^{(k)},\thetavs) \), while
\( \VP_{k}^{2} \) is the covariance matrix of the score
\( \nabla L(W^{(k)},\thetavs) \) under the parametric assumption
\( \fs \equiv \fs_{\thetavs} \).
In the global parametric situation, these two matrices coincide.
The value \( N_{k} \) can be treated as the local sample size corresponding to the
localizing scheme \( W^{(k)} \).

The following conditions will be assumed for our results.

\begin{itemize}
    \item[\( \Ass.1 \)] \( \{ Y_{i} \}^{n}_{i=1} \) are independent.

    \item[\( \Ass.2 \)]
    For some constants \( 0 < \nud < \nuu < 1 \),
\begin{EQA}[c]
    0 < \nud
    \le
    \| \DP_{k}^{-1} \DP_{k-1}^{2} \DP_{k}^{-1} \|_{\infty}
    \le
    \nuu < 1.
\label{u0uWk}
\end{EQA}
    \item[\( \Ass.3 \)]
    For a constant \( \fis > 0 \) and all \( k=1,\ldots,\K \), it holds
\begin{EQA}[c]
    \VP_{k}^{2} \le \fis^{2} \DP_{k}^{2} .
\label{VP2Wk}
\end{EQA}

    \item[\( \Ass.4 \)] For some fixed \( \delta < 1/2 \) and \( \rho > 0 \),
\begin{EQA}[c]
    \bigl| \logdens_{i}(u) / \logdens_{i}(0) - 1 \bigr| \le \delta,
    \qquad
    |u| \le \rho.
\label{piupi0}
\end{EQA}

    \item[\( \Ass.5 \)]
%
    The kernel function \( K_{\loc}(\cdot) \) is symmetric, \( K(0) = 1 \),
    \( K(u) \) decreases in \( u \ge 0 \) and \( K(u) = 0 \) for \( |u| \ge 1 \).
\end{itemize}

The condition
\( \Ass.2 \) effectively requires that the bandwidth sequence \( h_{k} \) grows
geometrically with \( k \).
Condition \( \Ass.3 \) is the local identifiability condition and it ensures that the
local variability of the process \( L(W^{(k)},\thetav) \) measured by the matrix \(
\VP_{k}^{2} \) is not significantly larger than the local information measured by the
matrix \( \DP_{k}^{2} \).
\( \Ass.4 \) only requires that the density functions \( \logdens_{i} (\cdot)\) are uniformly
continuous in a vicinity of zero.
In particular, the residuals can be unequally distributed.
All the results below tacitly assume that the conditions \( \Ass.1 \)--\( \Ass.5 \) hold.

Below we use generic notation \( C = C(\Ass) \) to indicate that a constant \( C \)
only depends on the constants from conditions \( \Ass.1 \)--\( \Ass.5 \) like \(
\fis \), \( \rho \), \( \delta \), \( \nud \), \( \nuu \), etc. We will also use
conditions  \( {(E\rr)} \), \( {(\cc{L}\rr)} \) etc. defined later in section
\ref{subcon}.


\subsection{Uniform concentration of the MLEs \( \tilde{\thetav}_{k}(x) \) under
\( \P_{\thetavs} \)}
The first result explains the localization property of the estimators
\( \tilde{\thetav}_{k}(x) \) from (\refeq{eq:nonk}) under the linear parametric structure
of the quantile function, that is, \( \fs(X_{i}) = \Psi_{i}^{\T} \thetavs \).
With some value \( \rups \) fixed, define for each \( k \le \K \) a local elliptic set
\begin{EQA}[c]
    \Theta_{k}(\rups)
    \eqdef
    \bigl\{ \thetav:  \| \VP_{k} (\thetav - \thetavs) \| \le \rups \bigr\}
\label{Thetasrups}
\end{EQA}
with \( \VP^{2}_{k} \) from \eqref{VPWqr}.
The question under study is a proper choice of the radius \( \rups \) which ensures
a prescribed small deviation probability for the event
\( \tilde{\thetav}_{k}(x) \not\in \Theta_{k}(\rups) \) uniformly in \( k \le \K \).

\begin{theorem}
\label{TLDqr}
Suppose \( {(E\rr)} \) and \( {(\cc{L}\rr)} \), and there exist constants \( C_{1} = C_{1}(\Ass) \) and \( C_{2} = C_{2}(\Ass) \)
such that the conditions
\begin{EQA}[c]
    \rups^{2} \ge C_{1} (\xx + \dimp+1),
    \qquad
    \rho^{2} N_{k} \ge C_{2} (\xx + \dimp+1)
\label{rupsconrqr}
\end{EQA}
ensure for \( k \le \K \)
\begin{EQA}
    \P_{\thetavs}\bigl\{ \tilde{\thetav}_{k}(x) \not\in \Theta_{k}(\rups) \bigr\}
    & \le &
    2 \ex^{-\xx} ,
\label{Pttexxx}
    \\
    \E_{\thetavs} \bigl[ L^{r}\bigl( W^{(k)},\tilde{\thetav}_{k}(x),\thetavs \bigr)
        \Ind\bigl\{ \tilde{\thetav}_{k}(x) \not\in \Theta_{k}(\rups) \bigr\} \bigr]
    & \le &
    C(\Ass) \ex^{-\xx} .
\end{EQA}
\end{theorem}
In particular, a choice \( \xx = \log (\K) + \xx_{0} \) and then
\( \rups^{2} \ge C_{1} (\xx + \dimp+1) \)
ensures a dominating probability
\( 1 - 2 \ex^{-\xx_{0}} \) for the joint concentration event
\begin{EQA}[c]
    \RS_{1}
    =
    \bigcap_{k=1}^{\K} \bigl\{ \tilde{\thetav}_{k}(x) \in \Theta_{k}(\rups) \bigr\} .
\label{cuptkTk}
\end{EQA}
In what follows we suppose that the values \( \xx = \log(\K) + \xx_{0} \) and \(
\rups \) are fixed in a way that the probability of the set \( \RS_{1} \) is
sufficiently close to 1.
This allows us to restrict ourselves to the case when each estimator \(
\tilde{\thetav}_{k}(x) \) belongs to the local vicinity \( \Theta_{k}(\rups) \).
The conditions in \eqref{rupsconrqr} require that \( \rups^{2} \) is of the order \(
\log(\K) + (\dimp+1) \), and the local sample size \( N_{k} \) should be at least of the
same order.
This conclusion is in agreement with our numerical simulation results (not reported
here). An increase of the polynomial degree \( \dimp \) requires the increase of the
smallest bandwidth \( h_{1} \) approximately by factor \( \dimp + 1 \) for
getting table behavior of the method.

\subsection{Uniform quadratic approximation of the local excess}
The previous subsection stated that the chance for any of the estimator
\(\tilde{\thetav}_{k}(x) \) lying outside the neighborhood \(\Theta_{k}(\rups)\) is
small, therefore in this subsection, we focus on the stochastic behavior of
\(\tilde{\thetav}_{k}\) in \(\Theta_{k}(\rups)\).
The proposed estimation procedure is likelihood-based:
all quantities are defined in terms of the quasi log-likelihood functions
\( L(W^{(k)},\thetav) \).
Particularly, the properties of the \emph{excess}
\( L\bigl( W^{(k)},\tilde{\thetav}_{k}(x),\thetavs \bigr)
\eqdef L\bigl( W^{(k)},\tilde{\thetav}_{k}(x) \bigr) - L(W^{(k)},\thetavs) \)
play a very important role in the whole method.
The famous Wilks result claims that the excess is asymptotically
\( \chi^{2}_{\dimp+1} \).
Unfortunately the local parametric approach for a narrow local neighborhood
of the point \( x \) leads to a relatively small effective sample size \( N_{k} \),
and the asymptotic results cannot be validated.
The general parametric approach of \cite{spo:11} though allows to operate
with finite samples and it can be directly applied
to a local parametric analysis.

It holds
\begin{EQA}
    \nabla L( W^{(k)}, \thetavs)
    &=& - \sum_{i=1}^{n} \rho'_{\tau}(Y_{i} - \Psi_{i}^{\T} \thetavs) w_{i}^{(k)}
    \\
    &=&
    \sum_{i=1}^{n} \bigl\{ - \tau  + \Ind(Y_{i} - \Psi_{i}^{\T} \thetavs < 0) \bigr\}
        \, \Psi_{i} \, w_{i}^{(k)} .
\label{LWktts}
\end{EQA}
Further, for \( \rd = (\rddelta,\rdomega) \) and
\( \DP_{k}^{2} \) from \eqref{DPWqr}, define
\begin{EQA}
    \DP_{\rdb,k}^{2}
    &=&
    \DP_{k}^{2} (1 - \rddelta) - \rdomega \VP_{k}^{2} ,
\label{DPrdk}
    \\
    \xiv_{\rdb,k}
    & \eqdef &
    \DP_{\rdb,k}^{-1} \, \nabla L( W^{(k)}, \thetavs) ,
\label{xivrdk}
\end{EQA}
and similarly for \( \rdm \eqdef - \rd = (-\rddelta,-\rdomega) \).
The values \( \rddelta, \rdomega \) are assumed to be small enough to ensure that
\( \DP_{\rdb,k}^{2} \) is positive and the value
\begin{EQA}[c]
    \alp_{\rd,k}
    \eqdef
    \lambda_{\max}
        \bigl( \Id_{\dimp+1} - \DP_{\rd,k} \DP_{\rdm,k}^{-2} \DP_{\rd,k} \bigr)
\label{alprdk}
\end{EQA}
is small as well.
Finally, define
\begin{EQA}
    \Lab(W^{(k)}, \thetav, \thetavs)
    & \eqdef &
    (\thetav -\thetavs)^{\T} \nabla L( W^{(k)}, \thetavs)
        - \frac{1}{2} \| \DP_{\rdb,k} (\thetav - \thetavs)\|^{2}
    \\
    & = &
    \xiv_{\rdb,k}^{\T} \DP_{\rdb,k} (\thetav -\thetavs)
        - \frac{1}{2} \| \DP_{\rdb,k} (\thetav - \thetavs)\|^{2}
\label{Labqr}
\end{EQA}
and a similar definition for \( \Lam(W^{(k)}, \thetav, \thetavs) \).
\begin{theorem}
\label{Twilksqr}
Under the conditions \( {(E\!D_{0})} \), \( {(E\!D_{1})} \), \( {(\LL_{0})}\),
it holds for all \( k \le \K \) and all \( \thetav \in \Theta_{k}(\rups) \)
\begin{EQA}[c]
    \Lam(W^{(k)}, \thetav, \thetavs) - \err_{\rdb,k}
    \le
    L(W^{(k)}, \thetav, \thetavs)
    \le
    \Lab(W^{(k)}, \thetav, \thetavs) + \err_{\rdb,k} \, .
\label{LamLLabqr}
\end{EQA}
Here \( \err_{\rdb,k} \) are the random error terms which fulfill
with some \( C_{1}(\Ass) \) and \( C_{2}(\Ass) \)
the following conditions: for any \( \xx > 0 \)
with \( C_{1}(\Ass) \xx + C_{2}(\Ass) \le \yyc \)
\begin{EQA}
    \P_{\thetavs}\bigl( \rdomega^{-1} \err_{\rdb,k} > C_{1}(\Ass) \xx + C_{2}(\Ass) (\dimp+1) \bigr)
    & \le &
    C(\Ass) \ex^{-\xx} ,
    \\
    \E_{\thetavs} \bigl| \rdomega^{-1} \err_{\rdb,k} \bigr|^{r}
    \le
    C_{r}(\Ass),
\label{Pthexyyqr}
\end{EQA}
where \( \yyc \) is a constant of order \( N_{k} \).
\end{theorem}

The sandwiching result \eqref{LamLLabqr} for each \( k \) follows from Theorem~3.1
of \cite{spo:11}.
It is only worth mentioning that the local sets \( \Theta_{k}(\rups) \) are
embedded:
\( \Theta_{1}(\rups) \supset \Theta_{2}(\rups)\supset \ldots \supset \Theta_{\K}(\rups) \), so it
suffices to check the bound \eqref{LamLLabqr} on \( \Theta_{1}(\rups) \) for each
\( k \le \K \).

The majorization bound \eqref{LamLLabqr} yields  that the maximum
of the process \( L(W^{(k)},\thetav,\thetavs) \) is also sandwiched between
the maximum of \( \Lab(W^{(k)}, \thetav, \thetavs) \) and of
\( \Lam(W^{(k)}, \thetav, \thetavs) \) up to a small random error term.
Moreover, \( \Lab(W^{(k)}, \thetav, \thetavs) \) is quadratic in \( \thetav \),
and its maximum is given by a quadratic form
\( \| \xiv_{\rdb,k} \|^{2}/2 \); similarly for \( \Lam(W^{(k)}, \thetav, \thetavs) \).
The next result presents a probabilistic bound for such quadratic forms.

\begin{theorem}
\label{Txivbound}
Assume \( \Ass.1 \) through \( \Ass.5 \).
There exist \( C_{1}(\Ass) \) and \( C_{2}(\Ass) \) such that for each \( \xx \)
with \( C_{1}(\Ass) \xx + C_{2}(\Ass) (\dimp+1) \le \yyc \) and \( k \le \K \), it holds
\begin{EQA}[c]
    \P_{\thetavs}\bigl\{
        \| \xiv_{\rdb,k} \|^{2} > C_{1}(\Ass) \xx + C_{2}(\Ass) (\dimp+1)
    \bigr\}
    \le
    2 \ex^{-\xx} .
\label{Pxivrdkqr}
\end{EQA}
Furthermore, for \( r > 0 \) and \( k \le \K \), it holds
\begin{EQA}[c]
    \E \| \xiv_{\rdb,k} \|^{2r}
    \le
    C_{r}(\Ass) \, .
\label{Exivrdkqr}
\end{EQA}
\end{theorem}

Consider the random set
\begin{EQA}[c]
    \RS_{2}
    =
    \bigcap_{k=1}^{\K} \bigl\{ \| \xiv_{\rdb,k} \| \le \rups \bigr\}.
\label{RS2qr}
\end{EQA}
Due to the bound of Theorem~\ref{Txivbound}, the choice
\( \rups^{2} = C_{1}(\Ass) (\xx + \log \K) + C_{2}(\Ass) (\dimp+1) \)ensures
that the probability of the set \( \RS_{2} \) is at least \( 1 - 2 \ex^{-\xx} \).

Below we restrict ourselves to the set \( \RS \) with
\( \RS = \RS_{1} \cap \RS_{2} \).
By construction
\begin{EQA}[c]
    \P(\RS)
    \ge
    1 - 4 \ex^{- \xx}
\label{PRS14xxx}
\end{EQA}
and on this set \( \tilde{\thetav}_{k} \in \Theta_{k}(\rups) \) and
\( \| \xiv_{\rdb,k} \| \le \rups \) for all \( k \le \K \).

The results of Theorem~\ref{Twilksqr} and \ref{Txivbound} have a number of important
corollaries; cf. \cite{spo:11}.

\begin{corollary}
\label{CTwilksqr}
It holds on \( \RS \) for every \( k \le \K \)
\begin{EQA}[c]
    \frac{1}{2} \| \xiv_{\rdm,k} \|^{2} - \err_{\rd,k}
    \le
    L(W^{(k)},\tilde{\thetav}_{k}(x),\thetavs)
    \le
    \frac{1}{2} \| \xiv_{\rdb,k} \|^{2} + \err_{\rdb,k} .
\label{LWkttqr}
\end{EQA}
\end{corollary}

\begin{corollary}
\label{CTwilksqr2}
It holds on \( \RS \) for every \( k \le \K \)
\begin{EQA}
    \bigl\|
        \DP_{\rdb,k} \bigl( \tilde{\thetav}_{k}(x) - \thetavs \bigr) - \xiv_{\rdb,k}
    \bigr\|^{2}
    & \le &
    4 \err_{\rdb,k} + \alp_{\rd,k} \| \xiv_{\rdb,k} \|^{2} ,
    \\
    \bigl\|
        \DP_{\rdb,k} \bigl( \tilde{\thetav}_{k}(x) - \thetavs \bigr)
    \bigr\|
    & \le &
    2 \err_{\rdb,k}^{1/2} + \bigl( 1 + \alp_{\rd,k}^{1/2} \bigr) \| \xiv_{\rdb,k} \| .
\label{DPktttqr}
\end{EQA}
\end{corollary}

The result of Corollary~\ref{CTwilksqr} can be viewed as a non-asymptotic version
of Wilks Theorem.
It claims that the twice excess \( 2L(W^{(k)},\tilde{\thetav}_{k}(x),\thetavs) \) can be
approximated by the quadratic form \( \| \xiv_{\rdb,k} \|^{2} \).
By definition \eqref{xivrdk}, each vector \( \xiv_{\rdb,k} \) is the normalized
score \( \nabla L(W^{(k)},\thetavs) \).
This score is a weighted and centered sum of Bernoulli random variables
\( \Ind(Y_{i} - \Psi_{i}^{\T} \thetavs < 0) \) with
\( \P_{\thetavs}\bigl( Y_{i} - \Psi_{i}^{\T} \thetavs < 0 \bigr) = \tau \); see \eqref{LWktts}.
So, its distribution under \( \P_{\thetavs} \) only depends on the design
\( X_{1},\ldots,X_{n} \) and on the weights \( w_{i}^{(k)} \).
This even applies to the joint distribution of all the \( \xiv_{\rdb,k} \) for
\( k=1,\ldots,\K \).
This important pivotality property explains that the computed critical values
\( \zz_{k} \) are almost independent of the underlying distribution of the errors
\( Y_{i} - f(X_{i}) \).
Further, by our identifiability condition (A3), \( \DP_{\rdb,k}^{2} \) is of the
same order as the variance
\( \VP_{k}^{2} = \Var\bigl\{ \nabla L(W^{(k)},\thetavs) \bigr\} \), so
\( \xiv_{\rdb,k} \)
is nearly normal under usual assumptions, thus the twice excess is asymptotically
\( \chi^{2}_{\dimp+1} \).

One can summarize the obtained general results as follows.
On the set \( \RS \) of dominating probability, each estimator \( \tilde{\thetav}_{k}(x) \)
belongs to the local vicinity \( \Theta_{k}(\rups) \) which yields
the bounds \eqref{LWkttqr}, \eqref{DPktttqr}.
Moreover, the random quantities \( \err_{\rdb,k} \) and \( \xiv_{\rdb,k} \) obey the
deviation and moment bounds of Theorem~\ref{Twilksqr} and Theorem~\ref{Txivbound}.


\subsubsection{Conditions from \cite{spo:11} \label{subcon}}
Here we list the conditions from \cite{spo:11} which are assumed to be fulfilled
for each local likelihood \( L(W^{(k)},\thetav) \), \( k \le \K \).
Some value \( \rups \) is assumed to be fixed for all conditions.
It separates the local zone of local quadratic approximation and the large
deviation zone.
The assumption are stated under the true data distribution \( \P \).
However, we apply the assumptions only in the case of linear parametric structure
with \( \fs(\cdot) \equiv \fs_{\thetavs}(\cdot) \).
Define
\begin{EQA}
    \zeta_{k}(\thetav)
    & \eqdef &
    L(W^{(k)},\thetav) - \E L(W^{(k)},\thetav)
    \\
    &=&
    -\sum_{i=1}^{n} \Bigl\{ \rho_{\tau}(Y_{i} - \Psi_{i}^{\T} \thetav)
        - \E \{ \rho_{\tau}(Y_{i} - \Psi_{i}^{\T} \thetav) \Bigr\} w_{i}^{(k)} .
\label{zetavkqr}
\end{EQA}
Also denote \( \nabla \zeta_{k}(\thetav) = \frac{d}{d\thetav} \zeta(\thetav) \).
The majorization bound \eqref{LamLLabqr} of Theorem~\ref{Twilksqr}
is stated in \cite{spo:11} under the following conditions.

\begin{description}
\item[\( \bb{(E\!D_{0})} \)]
    There exists a positive symmetric matrix \( \VP_{k}^{2} \),
    and constants \( \gm > 0 \) and \( \nunu \ge 1 \) such that
    \( \Var\bigl\{ \nabla \zeta_{k}(\thetav) \bigr\} \leq \VP_{k}^{2} \) and
    for all \( \lambda \) with \( | \lambda | \le \gm \),
\begin{EQA}[c]
    \sup_{\gammav \in \R^{\dimp+1}}
    \log \E \exp \biggl\{
        \lambda \frac{\gammav^{\T} \nabla \zeta_{k}(\thetavs)}{\| \VP_{k} \gammav\| }
    \biggr\}
    \le
    \nunu^{2} \lambda^{2} /2 ;
\label{eq:ed}
\end{EQA}
With this matrix \( \VP_{k} \), define the local set
\begin{EQA}[c]
    \Theta_{k}(\rups)
    =
    \{ \thetav: \| \VP_{k} (\thetav - \thetavs) \| \le \rups \} ;
\label{Thetakqr}
\end{EQA}

\item[\( \bb{(E\!D_{1})} \)]
    For each \( \rr \le \rups \), there exists a constant
    \( \rdomega(\rr) \le 1/2 \) such that it holds for all \( \thetav \in \Theta_{k}(\rups) \)
    and \(  | \lambda | \le \gm \):
\begin{EQA}[c]
    \sup_{\gammav \in \R^{\dimp+1}} \log \E \exp\biggl\{ \lambda
        \frac{\gammav^{\T} \{\nabla \zeta_{k}(\thetav) - \nabla \zeta_{k}(\thetavs)\}}
             {\rdomega(\rr) \| \VP_{k} \gammav \|}
    \biggr\}
    \le
    \nunu^{2} \lambda^{2}/2 ;
\end{EQA}

   \item[\( \bb{(\LL_{0})} \)]
    \textit{There are a positive matrix \( \DP_{k} \)
    and for each \( \rr \le \rups \) and a constant \( \rddelta(\rr) \le 1/2 \),
    such that it holds for all \( \thetav \in \Theta_{k} \)
    };
    \begin{EQA}[c]
\label{LmgfquadEL}
        \biggl|
            \frac{- 2 \E L(W^{(k)},\thetav,\thetavs)}
                 {\| \DP_{k} (\thetav - \thetavs) \|^{2}}
            - 1
        \biggr|
        \le
        \rddelta(\rr) ;
\end{EQA}

\item[\( \bb{(E\rr)} \)]
For any \( \rr \ge \rups \), there exist a value \( \gm(\rr) > 0 \) and a constant \( \nunu \)
    such that for all \( |\lambda| \le \gm(\rr) \),
\begin{EQA}[c]
    \sup_{\gammav \in \R^{\dimp+1}} \sup_{\thetav \in \Theta_k(\rr)}
    \log \E
        \exp\biggl\{
            \lambda \frac{\gammav^{\T} \nabla \zeta_{k}(\thetav)}{\| \VP_{k} \gammav \|}
        \biggr\}
    \le
    \nunu^{2} \lambda^{2}/2;
\end{EQA}

\item[\( \bb{(\LL{\rr})} \)]
    For each \( \rr \ge \rups \) and any \( \thetav \) with
    \( \| \VP_{k} (\thetav - \thetavs)\| = \rr \),
\begin{EQA}[c]
    \frac{- \E L(W^{(k)},\thetav,\thetavs)}
         {\| \VP_{k}(\thetav - \thetavs)\|^{2}}
    \ge
    \gmi(\rr) > 0.
\end{EQA}
\end{description}

All these conditions are assumed to be fulfilled for each \( k \le \K \).
Conditions\\ \( {(E\!D_{0})}, {(E\!D_{1})}\),\({(\LL_{0})}\) are local conditions which
should be applied on the local set \( \Theta_{k}(\rups) \), while
\( {(\LL{\rr})}, {(E\rr)} \) are global conditions which we apply on the complement
of \( \Theta_{k}(\rups) \).
Also \( {(E\!D_{0})}, {(E\!D_{1})}, {(E\rr)}\) are smoothness or moment assumptions
on the log likelihood process, and the conditions \({(\LL_{0})},{(\LL{\rr})}\)
ensure the identifiability properties.

\subsubsection{Proof of \( \bb{(E\rr)} \),  \( \bb{(E\!D_{0})} \) and \( \bb{(E\!D_{1})} \).}

Let us fix some \( k \le \K \).
Let \( N_{k} \) be the local sample size for the weighting scheme \( W^{(k)} \); see
\eqref{NWqr}.
Let also \( \rups \) by fixed in a way that
\( \rups |\Psi_{i}| \le \rho N_{k} \) for all
\( i \) with \( w_{i}^{(k)} > 0 \), that is, for all \( X_{i} \)
with \( |X_{i} - x| \le h_{k} \).

First we check \( {(E\rr)} \).
It holds by definition
\begin{EQA}
    \nabla \zeta_{k}(\thetav)
    &=&
    \sum_{i=1}^{n} \Psi_{i} \bigl[
        \Ind(Y_{i} - \Psi_{i}^{\T} \thetav < 0)
        - \P(Y_{i} - \Psi_{i}^{\T} \thetav < 0)
    \bigr] w_{i}^{(k)}
    \\
    &=&
    \sum_{i=1}^{n} \Psi_{i} \varepsilon_{i}(\thetav) w_{i}^{(k)}
\end{EQA}
with \( \varepsilon_{i}(\thetav) \eqdef \Ind(Y_{i} - \Psi_{i}^{\T} \thetav < 0)
- \P(Y_{i} - \Psi_{i}^{\T} \thetav < 0) \).
Obviously
\( \Ind(Y_{i} - \Psi_{i}^{\T} \thetav < 0) \) is a Bernoulli random variable with
the parameter \( \bern_{i}(\thetav) \eqdef \P(Y_{i} - \Psi_{i}^{\T} \thetav < 0) \)
and
\begin{EQA}
    \log \E \exp\{ \delta \varepsilon_{i}(\thetav)\}
    & = &
    \log\bigl\{
        1 - \bern_{i}(\thetav)
        + \bern_{i}(\thetav) \ex^{\delta}
     \bigr\}
    - \delta \bern_{i}(\thetav) .
\label{Eexepstpiqr}
\end{EQA}
The function \( g(\delta) \eqdef \log\bigl( 1 - q + q \ex^{\delta} \bigr) - q \delta \)
fulfills for any \( q < 1 \)
\begin{EQA}[c]
    g(0) = 0,
    \qquad
    g'(0) = 0,
    \qquad
    g''(\delta)
    \le q (1-q) \ex^{\delta}.
\label{f000de}
\end{EQA}
This implies
\begin{EQA}
    \log \E \exp\{ \delta \varepsilon_{i}(\thetav)\}
    & \le &
    \bern_{i}(\thetav) \{1-\bern_{i}(\thetav)\} \nunu^{2} \delta^{2}/2,
    \qquad
    |\delta| \le \gmiid
\end{EQA}
for a constant \( \nunu \ge 1 \)
depending on \( \gmiid \) only.
Therefore, it holds for any \( \gammav \in \R^{p+1} \) and \( \rho > 0 \) with
\( \rho  | \gammav^{\T} \Psi_{i}| \le \gmiid \) that,
\begin{EQA}
    \log \E \exp \{\rho \gammav^{\T} \nabla \zeta_{k}(\thetav)\}
    &\le&
    \log\E \exp\left\{
        \rho \sum_{i=1}^{n} \gammav^{\T} \Psi_{i}\varepsilon_{i}(\thetav) w_{i}^{(k)}
    \right\}
    \\
    &\le&
    \sum_{i=1}^{n}
        \log\E \exp\bigl\{
            \rho \gammav^{\T} \Psi_{i}\varepsilon_{i}(\thetav) w_{i}^{(k)}
        \bigr\}
    \\
    &\le&
    \sum_{i=1}^{n} \rho^{2}
        \bigl| \gammav^{\T} \Psi_{i} w_{i}^{(k)} \bigr|^{2} \bern_{i}(\thetav)
        \{ 1-\bern_{i}(\thetav) \} \nunu^{2}/2
    \\
    &\le&
    \nunu^{2} \rho^{2} \| \VP_{k}(\thetav) \gammav \|^{2}/2 ,
\end{EQA}
where
\begin{EQA}[c]
    \VP_{k}^{2}(\thetav)
    \eqdef
    \sum_{i=1}^{n}  \bern_{i}(\thetav) \{1-\bern_{i}(\thetav)\}
        \Psi_{i} \Psi_{i}^{\T} \bigl| w_{i}^{(k)} \bigr|^{2}.
\end{EQA}
This yields \( {(E\!D_{0})} \) with
\( \VP_{k}^{2} \eqdef \VP_{k}^{2}(\thetavs) \) and \( \gm = \gmiid N_{k}^{1/2} \);
see \eqref{NWqr}.
Furthermore, the linear PA \( \fs \equiv \fs_{\thetavs} \) yields
\( \bern_{i}(\thetavs) = \tau \) and hence
\begin{EQA}
    \VP_{k}^{2}(\thetavs)
    &=&
    4 {\tau (1 - \tau)} \VPb_{k}^{2}
\label{VPbk2qr}
\end{EQA}
for
\begin{EQA}
    \VPb_{k}^{2}
    & \eqdef &
    \frac{1}{4} \sum_{i=1}^{n} \Psi_{i} \Psi_{i}^{\T} \bigl| w_{i}^{(k)} \bigr|^{2}.
\label{VPbkdeqr}
\end{EQA}
For any \( \thetav \in \Theta \), it obviously holds
 \( \VP_{k}^{2}(\thetav) \le \VPb_{k}^{2}
\le \bigl\{ 4 \tau (1 - \tau) \bigr\}^{-1} \VP_{k}^{2} \), and thus
\( {(E\rr)} \) is fulfilled with
\( \gm^{2}(\rr) \equiv 4 \tau (1 - \tau) N_{k} \gmiid^{2} \).

Next we check the local condition \( {(E\!D_{1})} \).
For \( \rr \le \rups \) and \( \thetav \in \Theta_{k}(\rr) \), it holds
\begin{EQA}[c]
    \nabla\zeta_{k}(\thetav) - \nabla\zeta_{k}(\thetavs)
    =
    \sum_{i=1}^{n} \Psi_{i}
        \bigl\{ \varepsilon_{i}(\thetav) - \varepsilon_{i}(\thetavs) \bigr\} w_{i}^{(k)} .
\end{EQA}
Similarly to the above, the identity
\( \E \bigl\{ \varepsilon_{i}(\thetav) - \varepsilon_{i}(\thetavs) \bigr\}
= \bern_{i}(\thetav)- \bern_{i}(\thetavs) \) implies
\begin{EQA}
    && \nquad
    \log \E \exp\bigl[
        \lambda \gammav^{\T}
        \bigl\{ \nabla\zeta(\thetav) - \nabla\zeta(\thetavs) \bigr\}
    \bigr]
    \\
    & \le &
    2 \nunu^{2} \lambda^{2} \| \VPb_{k} \gammav \|^{2}
    \max_{i \le n} \bigl| \bern_{i}(\thetav)- \bern_{i}(\thetavs) \bigr|
        \Ind\bigl( w_{i}^{(k)} > 0 \bigr)
    \\
    & \le &
    \omega(\rr) \nunu^{2} \lambda^{2} \| \VPb_{k} \gammav \|^{2} /2
\end{EQA}
with
\begin{EQA}[c]
    \omega(\rr)
    \eqdef
    4 \max_{i \le n} \sup_{\thetav \in \Theta_{k}(\rups)}
        \bigl\{ \bern_{i}(\thetav)- \bern_{i}(\thetavs) \bigr\}
        \Ind\bigl( w_{i}^{(k)} > 0 \bigr).
\end{EQA}
Further, for any \( \thetav \in \Theta_{k}(\rr) \), it holds
\( \bigl| \Psi_{i}^{\T} (\thetav - \thetavs) \bigr| w_{i}^{(k)} \le \rr / N_{k} \)
\begin{EQA}
    |\bern_{i}(\thetav) - \bern_{i}(\thetavs)| \Ind\bigl( w_{i}^{(k)} > 0 \bigr)
    & \le &
    C | \Psi_{i}^{\T}(\thetav - \thetavs) | \Ind\bigl( w_{i}^{(k)} > 0 \bigr)
    \\
    & \le &
    C N_{k}^{-1/2} \| \VPb_{k}(\thetav - \thetavs)\|
    \le
    C  N_{k}^{-1/2} \rr,
\end{EQA}
and \( {(E\!D_{1})} \) holds with \( \rdomega(\rr) = N_{k}^{-1/2} \rr \).

\subsubsection{The \({(\LL_{0})}\) and \( {(\LL r)} \) conditions}
These identifiability conditions will be checked under the measure
\( \P_{\thetavs} \) corresponding to the linear quantile function
\( \fs(\cdot) = \fs_{\thetavs}(\cdot) \).
It holds
\begin{EQA}[c]
    \nabla \E_{\thetavs} L(W^{(k)},\thetav)
    =
    - \sum_{i=1}^{n} \Psi_{i} \Bigl\{
        \tau - \P\bigl( Y_{i}- \Psi_{i}^{\T} \thetav < 0 \bigr)
    \Bigr\} w_{i}^{(k)}
\end{EQA}
and
\begin{EQA}[c]
    - \nabla^{2} \E_{\thetavs} L(W^{(k)},\thetav)
    =
    \sum_{i=1}^{n} \Psi_{i} \Psi_{i}^{\T}
        \logdens_{i}\bigl( \Psi_{i}^{\T}(\thetav - \thetavs) \bigr) w_{i}^{(k)}
    \eqdef
    \DP^{2}_{k}(\thetav).
\end{EQA}
Now the Taylor expansion of \( - \E_{\thetavs} L(W^{(k)},\thetav) \) at
the extreme point \( \thetav = \thetavs \) implies
\begin{EQA}
    \E_{\thetavs} L(W^{(k)},\thetav) - \E_{\thetavs}L(W^{(k)},\thetavs)
    &=&
    - \frac{1}{2} \sum_{i=1}^{n}  | \Psi_{i}^{\T} (\thetav - \thetavs) |^{2}
        \logdens_{i}\bigl(  \Psi_{i}^{\T}(\thetavd-\thetavs) \bigr) w_{i}^{(k)}
    \\
    &=&
    - \frac{1}{2} (\thetav - \thetavs)^{\T} \DP_{k}^{2}(\thetavd)(\thetav - \thetavs)
\end{EQA}
for some \( \thetavd \in [\thetav,\thetavs] \).
Further, for any \( \thetav \in \Theta_k(\rups) \),
it holds by \( \Ass. 4\)
\begin{EQA}[c]
    \Bigl|
        \frac{\logdens_{i}\bigl(  \Psi_{i}^{\T} (\thetav - \thetavs) \bigr)}{\logdens_{i}(0)}
        - 1
    \Bigr|
    \Ind\bigl( w_{i}^{(k)} > 0 \bigr)
    \le
    \delta ,
\label{densdqr}
\end{EQA}
and \( (\LL_{0}) \) follows.
The global identifiability condition \( {(\LL r)} \) is fulfilled if
\( \rr^{2} \ge C_{1}  (\xx + \dimp+1) \) for some fixed constants \( C_{1} \);
see \cite{spo:11}, Section 5.3, for more details.


\subsection{Theorem for critical values}
\label{eq:theocrit}

The theorem below assures an upper bound for the critical values \( \zz_{k} \)
constructed in Section~\ref{SpropCV}.
To avoid technical burdens, we restrict the analysis to the random set \( \RS \) and
discard the large deviation probability part on its complement.
The notation \( \PR(B) \) for a set \( B \) means \( \P(B \cap \RS) \).

\begin{theorem}
Suppose that \( r>0,\alpha >0 \).
There exist constants \( a_{0},a_{1}\) s.t. the propagation
condition is fulfilled with the choice of
\begin{EQA}[c]
        \zz_{k}
        =
        a_{0} + \log(\alpha^{-1}) + a_{1} r (\K - k) + r \log((\dimp+1))
\label{eq:crit}
\end{EQA}
\end{theorem}

\begin{proof}
First we bound the quantity
\( L\bigl( W^{(k)},\tilde{\thetav}_{k}(x),\tilde{\thetav}_{\ell}(x) \bigr) \) on
the random set \( \RS = \RS_{1} \cap \RS_{2}\).
The majorization \eqref{LamLLabqr} and its corollary \eqref{LWkttqr} yield
on \( \RS \) with
\( \uv_{\ell k} \eqdef \DP_{\rdm,k} \bigl( \tilde{\thetav}_{\ell}(x) - \thetavs \bigr) \)
\begin{EQA}
    L\bigl( W^{(k)},\tilde{\thetav}_{k}(x),\tilde{\thetav}_{\ell}(x) \bigr)
    &=&
    L\bigl( W^{(k)},\tilde{\thetav}_{k}(x),\thetavs \bigr)
    -
    L\bigl( W^{(k)},\tilde{\thetav}_{\ell}(x),\thetavs \bigr) .
\label{LkLell}
    \\
    & \le &
    \frac{1}{2} \| \xiv_{\rdb,k} \|^{2}
    - \Lam\bigl( W^{(k)},\tilde{\thetav}_{\ell}(x),\thetavs \bigr) + \err_{\rdb,k}
\label{LkWktt}
    \\
    &=&
    \frac{1}{2} \| \xiv_{\rdb,k} \|^{2}
    - \uv_{\ell k}^{\T} \xiv_{\rdm,k} + \frac{1}{2} \| \uv_{\ell k} \|^{2}
    + 2 \err_{\rdb,k}
    \\
    & \le &
    \frac{1}{2} \bigl( \| \xiv_{\rdb,k} \| + \| \uv_{\ell k} \| \bigr)^{2}  + 2 \err_{\rdb,k}
    \\
    & \le &
    \| \xiv_{\rdb,k} \|^{2} + \| \uv_{\ell k} \|^{2} + 2 \err_{\rdb,k} ,
\label{supLamqr}
\end{EQA}
where we used the fact that \( \| \xiv_{\rdm,k} \| \le \| \xiv_{\rdb,k} \| \).
It is not difficult to see that
\begin{EQA}[c]
    \| \uv_{\ell k} \|^{2}
    =
    \| \DP_{\rdm,k} \DP_{\rd,\ell}^{-1} \DP_{\rd,\ell} \bigl( \tilde{\thetav}_{\ell} - \thetavs \bigr) \|^{2}
    \le
    \| \DP_{\rdm,k} \DP_{\rd,\ell}^{-2} \DP_{\rdm,k} \|_{\infty} \,\,
        \| \DP_{\rd,\ell} \bigl( \tilde{\thetav}_{\ell} - \thetavs \bigr) \|^{2} .
\label{uvellkqr}
\end{EQA}
By construction \( \DP_{\rdb,k}^{2} \le \DP_{k}^{2} \le \DP_{\rdm,k}^{2} \) and
the definition \eqref{alprdk}  implies by \( \alp_{\rd,k} \le 1/2 \)
\begin{EQA}[c]
    \DP_{\rdm,k}^{2}
    \le
    (1 - \alp_{\rd,k})^{-1} \DP_{\rdb,k}^{2}
    \le
    2 \DP_{\rdb,k}^{2} \, .
\label{DPrdbkrdm}
\end{EQA}
Now it follows from condition \( \Ass.2 \) that
\begin{EQA}
    \| \DP_{\rdm,k} \DP_{\rdb,\ell}^{-2} \DP_{\rdm,k} \|_{\infty}
    & \le &
    2 \| \DP_{k} \DP_{\ell}^{-2} \DP_{k} \|_{\infty}
    \le
    \begin{cases}
        2/\nud^{k - \ell}, & k > \ell, \\
        2 \nuu^{\ell-k}, & k < \ell.
    \end{cases}
\label{DPrdmkrdk}
\end{EQA}
Corollary~\ref{CTwilksqr2} implies
\begin{EQA}[c]
    \| \DP_{\rd,\ell} \bigl( \tilde{\thetav}_{\ell}(x) - \thetavs \bigr) \|
    \le
    2 \err_{\rdb,\ell}^{1/2}
        + \bigl( 1 + \alp_{\rd,\ell}^{1/2} \bigr) \| \xiv_{\rdb,\ell} \|
    \le
    2 \err_{\rdb,\ell}^{1/2} + 2 \| \xiv_{\rdb,\ell} \| .
\label{DPrdell}
\end{EQA}
We also use that \( \E_{\thetavs} \| \xiv_{\rdb,k} \|^{2r} \le (\dimp+1)^{r} C_{r}(\Ass) \) for all \( k \le \K \).
Now it holds from \eqref{supLamqr}, \eqref{DPrdmkrdk}, and \eqref{DPrdell}
for \( k > \ell \)
\begin{EQA}
    \ER_{\thetavs} L^{r}\bigl( W^{(k)},\tilde{\thetav}_{k}(x),\tilde{\thetav}_{\ell}(x) \bigr)
    & \le &
    \ER_{\thetavs} \left[ \| \xiv_{\rdb,k} \|^{2}
        + 8 \nud^{- k + \ell} \bigl( \err_{\rdb,\ell}^{1/2} + \| \xiv_{\rdb,\ell} \| \bigr)^{2}
    + 2 \err_{\rdb,k}  \right]^{r}
    \\
    & \le &
    C(\Ass)(\dimp+1)^{r}  \nud^{- r (k - \ell)} .
\label{LAAkl}
\end{EQA}
Similarly one can show that for \( k < \ell \) by \( \nuu < 1 \)
\begin{EQA}
    \ER_{\thetavs} L^{r}\bigl( W^{(k)},\tilde{\thetav}_{k}(x),\tilde{\thetav}_{\ell}(x) \bigr)
    & \le &
    \ER_{\thetavs} \bigl[
        \| \xiv_{\rdb,k} \|^{2} + 8 \bigl( \err_{\rdb,\ell}^{1/2} +  \| \xiv_{\rdb,\ell} \|
    \bigr)^{2}  + 2 \err_{\rdb,k} \bigr]^{r}
    \\
    & \le &
    C(\Ass) (\dimp+1)^{r} .
\label{LAAklm}
\end{EQA}
Also by Theorem~\ref{Txivbound} for \( \xx > 0 \)
\begin{EQA}[c]
    \P_{\thetavs} \bigl\{
        L\bigl( W^{(k)},\tilde{\thetav}_{k}(x),\tilde{\thetav}_{\ell}(x) \bigr)
        >
        C_{1} (\dimp+1) + C_{2} \xx
    \bigr\}
    \le
    2 \ex^{-\xx} .
\label{PRLDqr}
\end{EQA}
These bounds can be used to check that the critical value \( \zz_{k} \) which is selected
in the form \eqref{eq:crit} to ensure the propagation condition in \eqref{eq:propa}.
%
Consider a random set \( \cc{B}_{\ell} \eqdef \{ \hat{k}(x) = \ell \} \),
By definition of \( \hat{k} \), when \(\cc{B}_{\ell}\) happens,
at least one of the estimator \( \tilde{\thetav}_{\ell+1}(x) \) must be not accepted, that is,
\begin{EQA}[c]
    \cc{B}_{\ell}
    \subseteq
    \bigcup_{m=1}^{\ell} \Bigl\{
        L\bigl( W^{(m)},\tilde{\thetav}_{m}(x),\tilde{\thetav}_{\ell+1}(x) \bigr) > \zz_{m}
    \Bigr\} .
\label{Bellqr}
\end{EQA}
The bounds \eqref{LAAkl} and \eqref{PRLDqr} yield by the Cauchy-Schwarz inequality
\begin{EQA}
    && \nquad
    \ER_{\thetavs}
        L^{r}\bigl( W^{(k)},\tilde{\thetav}_{k}(x),\hat{\thetav}_{k}(x) \bigr)
    \\
    & \le &
    \sum_{\ell=1}^{k}
        \bigl[ \ER_{\thetavs}
            L^{2r}\bigl( W^{(k)},\tilde{\thetav}_{k}(x),\tilde{\thetav}_{\ell}(x)
            \bigr)
        \bigr]^{1/2}
        \bigl[ \PR_{\thetavs}(\cc{B}_{\ell}) \bigr]^{1/2}
    \\
    & \le &
    C(\Ass) (\dimp+1)^{2r}
    \sum_{\ell=1}^{k} \nud^{-2r(k - \ell)} \bigl[ \PR_{\thetavs}(\cc{B}_{\ell}) \bigr]^{1/2}
    \\
    & \le &
    C(\Ass) (\dimp+1)^{2r}
    \sum_{\ell=2}^{k} \nud^{-2r(k - \ell)} \left(
        \sum_{m=1}^{\ell} \PR_{\thetavs} \Bigl\{
            L\bigl( W^{(m)},\tilde{\thetav}_{m}(x),\tilde{\thetav}_{\ell+1}(x) \bigr)
            > \zz_{m}
        \Bigr\}
    \right)^{1/2} .
\end{EQA}
Fix \( c_{0} > \log (\nud^{-1}) \) and consider
\( \zz_{m} = C_{1} (\dimp+1) + C_{2} \xx_{m} \) with
\( \xx_{m} = 2 c_{0} r (\K - m) + 2 \xx  \) for some \( \xx \).
Then \eqref{PRLDqr} implies
\begin{EQA}
    && \nquad
    \ER_{\thetavs} \bigl[
        L^{r}\bigl( W^{(k)},\tilde{\thetav}_{k}(x),\hat{\thetav}_{k}(x) \bigr)
    \bigr]
    \\
    & \le &
    C(\Ass) (\dimp+1)^{2r} \sum_{\ell=2}^{\K} \nud^{-2r(\K - \ell)}
    \biggl( \sum_{m=1}^{\ell} 2 \ex^{- \xx_{m}} \biggr)^{1/2}
    \\
    & \le &
    C(\Ass) (\dimp+1)^{2r} \ex^{-\xx}
        \sum_{\ell=2}^{\K} \exp\bigl[ -2 r (\K - \ell) \bigl\{ c_{0} - \log(1/\nud) \bigr\} \bigr]
    \\
    & \le &
    C(\Ass) (\dimp+1)^{2r} \ex^{-\xx}
\label{ERtsum12}
\end{EQA}
and the bound \eqref{eq:propa} follows with
\( \xx = \log(1/\alp) + r \log(\dimp+1) + a_{0} \) for a proper \( a_{0} \).
\end{proof}

\subsection{Propagation Property and Stability}
The oracle result is a consequence of two properties of the procedure:
propagation under homogeneity and stability.
The first one means that the algorithm does not terminate for \( k < \ks \)
(no false alarm) with a high probability.
The stability property ensures that the estimation quality will not essentially
deteriorate in the steps ``after propagation'' for \( k > \ks \).

By construction, the procedure described in Section 2 provides the prescribed
performance if the true quantile function \( f(\cdot) \) follows the parametric
model:
at any intermediate step \( k < \K \) the non-adaptive estimator
\( \tilde{\thetav}_{k}(x) \) and the adaptive estimator \( \hat{\thetav}_{k}(x) \)
coincide with high probability
yielding that
\( \E_{\thetavs} L^{r}\bigl( W^{(k)},\tilde{\thetav}_{k}(x),\hat{\thetav}_{k}(x) \bigr) \)
is small.
The next theorem claims a similar performance of the \( k \) step estimator
\( \hat{\thetav}_{k}(x) \)
under the true nonparametric model \( \fs(\cdot) \),
however, the propagation property is only guaranteed for
\( k \le \ks \), that is, while the SMB assumption is fulfilled.

\begin{theorem}
Assume \( \Delta_{\ks}(\thetav) \le \Delta \) for some \( \ks \).
Then for any \( k \le \ks \)
\begin{EQA}
    \E \log\bigl\{
        1 + L^{r}\bigl( W^{(k)},\tilde{\thetav}_{k}(x),\hat{\thetav}_{k}(x) \bigr)
        / \riskt_{r}
    \bigr\}
    & \le &
    \Delta + \alpha,
\label{eq:prog1}
\end{EQA}
\end{theorem}

The bound \eqref{eq:prog1} can be derived
from the next general result; see \cite{Spokoiny:2009}.

\begin{lemma}
\label{Linftheor}
Let \( \P \), \( \P_{0} \), be two measures s.t. \( \E \log(d\P/d\P_{0}) \le \Delta <
\infty \).
For any random variable \( Z \) with \( \E_{0} Z < \infty \), it holds
\( \E \log(1+Z) \le \Delta + \E_{0} Z \).
\end{lemma}

%

The propagation result \eqref{eq:prog1} explains well the behavior of the procedure
for the first \( \ks \) steps.
In addition, we also need a \emph{stability property} which makes sure that
at the further steps of the algorithm for \( k > \ks \),
the quality of the obtained adaptive estimator \( \hat{\thetav}_{k}(x) \)
will not significantly deteriorate.
The stability property can be stated as follows.

\begin{theorem}
The adaptive estimator \( \hat{\thetav}(x) \) fulfills
\begin{EQA}[c]
\label{LWkshatk}
    L\bigl( W^{(\ks)},\tilde{\thetav}_{\ks}(x),\hat{\thetav}(x) \bigr)
        \Ind\bigl\{ \hat{k}(x) > \ks \bigr\}
    \le
    \zz_{\ks}.
\end{EQA}
\end{theorem}

Due to \eqref{LWkshatk}, on the set \( \{ \hat{k}(x) \ge \ks \} \), the adaptive estimator
\( \hat{\thetav}(x) \) belongs to the confidence set \( \CS_{\ks}(\zz_{\ks}) \) of the oracle
estimator \( \tilde{\thetav}_{\ks}(x) \).
This assertion follows from the setup of our procedure because the estimate
\( \hat{\thetav}(x) = \tilde{\thetav}_{\hat{k}(x)}(x) \) is accepted.
If
\( \hat{k}(x) > \ks \), it should be consistent with \( \tilde{\thetav}_{\ks}(x) \),
and thus it belongs to the confidence set of
\( \tilde{\thetav}_{\ks(x)}(x) \).

\subsection{Proof of the ``oracle" property}
The \emph{propagation} and \emph{stability} results yield
\begin{EQA}
    && \nquad
    \E \log\bigl\{ 1 +
        \riskt_{r}^{-1}
        L^{r}\bigl( W^{(\ks)},\tilde{\thetav}_{\ks}(x),\tilde{\thetav}_{\hat{k}}(x) \bigr)
    \bigr\}
    \\
    &=&
    \E\left[
        \log\bigl\{
            1 + \riskt_{r}^{-1} L^{r}\bigl( W^{(\ks)},\tilde{\thetav}_{\ks}(x),\tilde{\thetav}_{\hat{k}}(x) \bigr)
        \bigr\} \Ind(\hat{k}\le \ks)
    \right]
    \\
    && \qquad
    + \, \E \left[ \log\left\{
            1 + \riskt_{r}^{-1} L^{r}\bigl( W^{(\ks)},\tilde{\thetav}_{\ks}(x),\tilde{\thetav}_{\hat{k}}(x) \bigr)
        \right\} \Ind(\hat{k} > \ks)
    \right]
    \\
    & \le &
    \Delta + \E_{\thetavs}\left[
        \riskt_{r}^{-1} L^{r}\bigl( W^{(\ks)},\tilde{\thetav}_{\ks}(x),\tilde{\thetav}_{\hat{k}}(x) \bigr)
    \right]
    \\
    && \qquad
    + \, \E \log\left\{
        1 + \riskt_{r}^{-1} L^{r}\bigl( W^{(\ks)},\tilde{\thetav}_{\ks}(x),\tilde{\thetav}_{\hat{k}}(x) \bigr)
        \Ind(\hat{k} > \ks)
    \right\}
    \\
    & \le &
    \Delta + \rho + \log(1+\zz_{\ks}/\riskt_{r})
\end{EQA}

 \bibliography{bibloc}
 \typeout{get arXiv to do 4 passes: Label(s) may have changed. Rerun}

\end{document}